\documentclass[12pt,reqno]{amsart}
\usepackage[T1]{fontenc} 
\usepackage[margin=1in]{geometry}
\usepackage{graphicx}
\usepackage{amssymb,amsmath,amsthm}
\usepackage[all]{xy}
\usepackage{multicol}
\usepackage{amsfonts,amsmath, tikz}
\usetikzlibrary{topaths,calc}
\usepackage{leftidx}
\usepackage{tikz, tikz-cd}
\usetikzlibrary{topaths,calc}
\usepackage{mathrsfs}
\usepackage{lscape}
\usepackage{setspace}
\usepackage{wrapfig}

\newtheorem{theorem}{Theorem}[section]
\newtheorem{proposition}[theorem]{Proposition}

\newtheorem{lemma}[theorem]{Lemma}

\theoremstyle{definition}
\newtheorem{definition}[theorem]{Definition}
\newtheorem{remark}[theorem]{Remark}
\numberwithin{equation}{section}
\newtheorem{fig}[theorem]{Figure}
\newtheorem{example}[theorem]{Example}

 \usepackage[pagewise]{lineno}

\begin{document}

\title[Scattering Symplectic Geometry]{Symplectic, Poisson, and Contact Geometry on Scattering Manifolds} 

\author{Melinda Lanius}
\address{Department of Mathematics, University of Arizona, 617 North Santa Rita Avenue, Tucson, AZ 85721, USA}
\email{lanius@math.arizona.edu}

\begin{abstract}  We introduce scattering-symplectic manifolds, manifolds with a type of minimally degenerate Poisson structure that is not too restrictive so as to have a large class of examples, yet restrictive enough for standard Poisson invariants to be computable. This paper will demonstrate the potential of the scattering symplectic setting. In particular, we construct scattering-symplectic spheres and scattering symplectic gluings between strong convex symplectic fillings of a contact manifold. By giving an explicit computation of the Poisson cohomology of a scattering symplectic manifold, we also introduce a new method of computing Poisson cohomology. 
\end{abstract}
\maketitle

\section{Introduction}

To study the standard Euclidean symplectic form near infinity, one can use scattering-symplectic geometry. To be precise, let $(p_1,q_1,\dots,p_n,q_n)$ be the standard coordinates on $\mathbb{R}^{2n}$ and $$r =\sqrt{p_1^2+q_1^2+\dots+p_n^2+q_n^2}$$ be the radial coordinate. Away from the origin, define a function $x=\dfrac{1}{r}$. We can compactify $\mathbb{R}^{2n}$ with a sphere $\mathbb{S}^{2n-1}$ at infinity that is given by the zero set of $x$. Then in coordinates $t_i=p_ix$ and  $s_i=q_ix$ the standard Euclidean symplectic form  is expressible as \begin{align*}\omega &= \sum_{i=1}^n dp_i\wedge dq_i\\ &=\sum_{i=1}^n\dfrac{dx}{x^3}\wedge (s_idt_i-t_ids_i)+\frac{1}{x^2}dt_i\wedge ds_i\\ &=\frac{dx}{x^3}\wedge\alpha-\frac{d\alpha}{2x^2} \end{align*} where $\alpha= s_idt_i-t_ids_i$ defines the standard contact structure on the sphere $\mathbb{S}^{2n-1}$.  This compactified space equipped with $\omega$ is an example of a scattering-symplectic manifold (with boundary). In fact, we will later show that any scattering-symplectic form $\omega$ on any manifold $M$ will define a contact structure on the singular locus of $\omega$ (i.e. the zero set of $x$). Consqeuntly, scattering-symplectic geometry has several interesting connections to questions arising in contact geometry. If we invert $\omega$, then we have a Poisson bivector that is non-degenerate away from the boundary sphere. Generally, studying scattering-symplectic structures is equivalent to understanding a new class of mildly degenerate Poisson structures. The purpose of this paper is to introduce this class of mildly-degenerate Poisson structures and to discuss our approach for studying them. \\

\noindent {\bf Organization.} In Section 2, we discuss the general framework that we will use for our in-depth study of scattering-symplectic structures. While many of the results in this section are not novel, our approach to the proofs brings techniques from microlocal analysis to bear in the study of mildly-degenerate Poisson structures. As is frequently the case in mathematics, the right tools can make a seemingly hard problem quite tractable. Our design is that others will find these tools very useful in understanding their own structures of interest.  In Section 3, we delve into scattering-symplectic geometry. Below we recount some of the exciting features of this new area. 

\subsection{ \bf Why geometric microlocal analysis? Why scattering-symplectic geometry?} 

Our work was inspired by the $b$-symplectic setting, introduced by Victor Guillemin, Eva Miranda, and Ana Rita Pires \cite{guillemin2014symplectic}. In the same way a non-degenerate Poisson structure corresponds to a sympelctic form, Guillemin, Miranda, and Pires view a type of mildly degenerate Poisson structure as non-degenerate on the $b$-tangent bundle, a vector bundle introduced by Richard Melrose in his study of cylindrical metrics (See Chapter 7 \cite{Melrose95}). This so-called $b$-Poisson structure corresponds to a b-symplectic form. The $b$-setting, in particular the $b$-tangent bundle, has its genesis in a compactification of a manifold with a cylindrical end in that the $b$-tangent bundle is a Lie algebroid that extends the standard tangent bundle to this compactification.

\begin{fig} The left torus has a cylindrical end. The right has a Euclidean end.

\begin{center}

\includegraphics[scale=.35]{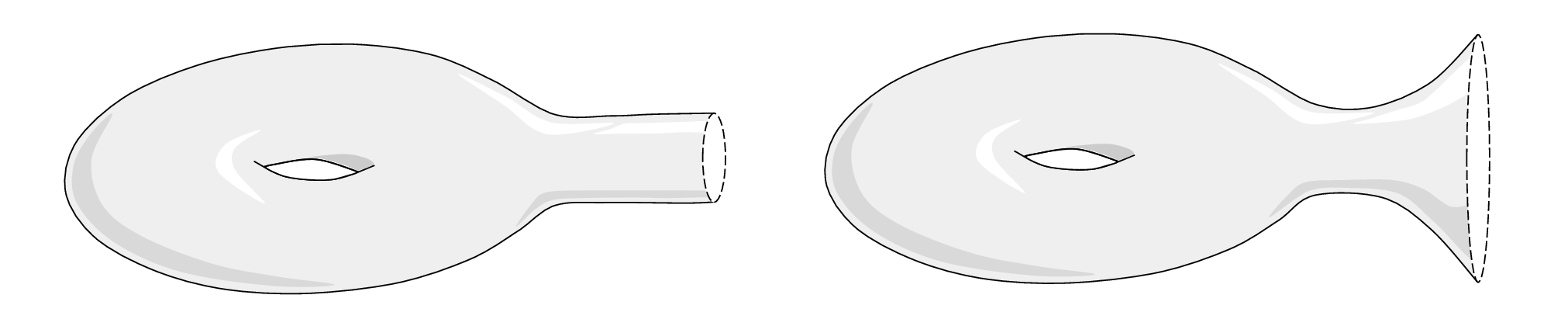}

{\bf $b$-manifold \hspace{15ex} sc-manifold \hspace{4ex} ~}

\end{center} 

\end{fig}

Scattering metrics, or metrics that behave like the funnel end of a Euclidean cone, are also extensively studied by Richard Melrose. His scattering-tangent bundle arises in a compactification of a manifold with a Euclidean end and is similarly acceptable for forming a correspondence between a non-degenerate Poisson structure on the scattering-tangent bundle and a scattering-sympelctic form.\\

\subsection{Main Results}~\\

\noindent {\bf \emph{Simple Poisson Structures on Spheres.}} Because every closed symplectic manifold has non-trivial degree 2 de Rham cohomology group, the only sphere that admits a symplectic structure is the 2 dimensional sphere $\mathbb{S}^2$. An immediate question when expanding the notion of symplectic is to look for these structures on spheres. 
 
\begin{theorem} Every even dimensional sphere $\mathbb{S}^{2n}$ admits a scattering symplectic structure $\omega$ such that the equator $\mathbb{S}^{2n-1}$ is the singular hypersurface and $\omega$ induces the standard contact structure on $\mathbb{S}^{2n-1}$.\end{theorem} 

One very natural context that provides a generalized symplectic structure on spheres is folded symplectic geometry. Richard Melrose \cite{melrose1981equivalence}, by defining folded contact structures, introduced a particular idea of minimally degenerate differential form. Building on \cite{melrose1981equivalence}, Ana Cannas da Silva, Victor Guillemin, and Christopher Woodward \cite{da2000unfolding} define a folded symplectic manifold $(M^{2n},Z,\omega)$ to be a $2n$-dimensional manifold $M$ equipped with a closed two-form $\omega$ that is non-degenerate except on a hypersurface $Z$ where there exist coordinates such that locally $Z=\left\{x_1=0\right\}$ and 

\centerline{$\displaystyle{\omega=x_1dx_1\wedge dy_1+\sum_{i=2}^n dx_i\wedge dy_i}.$}

By allowing this very mild degeneracy, all even dimensional spheres admit a folded symplectic structure. Folded symplectic geometry is similar in spirit to our main object of study. One significant difference is that scattering-symplectic structures correspond to classical minimally-degenerate Poisson structures. The inverse of the morphism $\omega^\flat:TM\to T^*M$ (where it is invertible) for a folded symplectic form does not extend to define a Poisson structure on the entire manifold. 

Unfortunately, there are no b-Poisson spheres in dimensions greater than two: Ioan M\u{a}rcut, and Boris Osorno Torres \cite{muarcuct2014cohomological} showed that a compact $b$-symplectic manifold $(M, Z)$ of dimension $2n$ has a class $c$ in $H^2(M)$ such that $c^{n-1}$ is nonzero in
$H^{2n-2}(M)$. \\

\noindent {\bf \emph{Symplectic Gluing.}} Recall that a symplectic manifold $(M,\omega)$ is a strong symplectic filling of a contact manifold $(Z,\xi)$  if $Z$ is the boundary of $M$ and near $Z$ there is a Liouville vector field $V$ transverse to $Z$ with $i_V\omega$ defining the contact structure $\xi$ such that $\mathcal{L}_V\omega=\omega$. 
These fillings come in two flavors: \emph{convex} means the Liouville vector field $V$ points outward at the boundary $Z$ and \emph{concave} means $V$ points inward at the boundary. 

In order to glue two strong symplectic fillings along a common contact boundary to form a symplectic manifold, one side must be a concave filling and the other must be convex. In this paper, we demonstrate a natural way to expand the symplectic category to allow gluings of convex to convex and concave to concave fillings. 

\begin{theorem} Given two strong convex symplectic fillings of a contact manifold, their union over $Z$ admits a scattering symplectic structure that coincides with the existing symplectic structures outside of a collar
neighborhood of  $Z$.

Given two strong concave symplectic fillings of a contact manifold, their union over $Z$ admits a folded symplectic structure that coincides with the existing symplectic structures outside of a collar
neighborhood of $Z$. \end{theorem}

\begin{fig} On the left, convex to convex produces scattering symplectic. On the right, concave to concave gives folded symplectic.

\begin{center}\includegraphics[scale=.65]{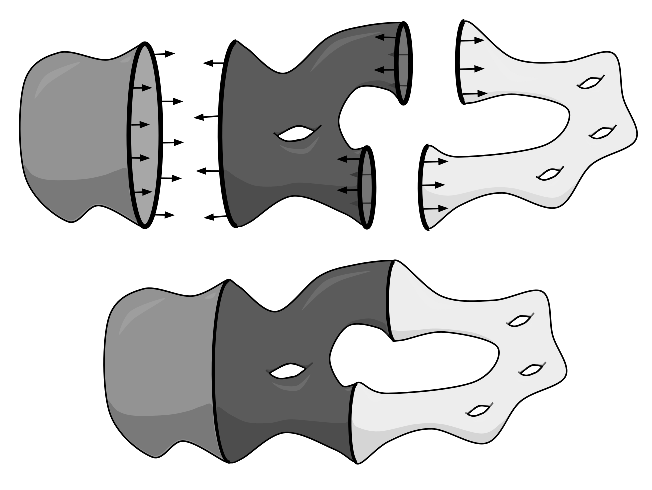}\end{center}

 \end{fig}

\vspace{2ex} 

 We can use this theorem to construct many examples of scattering-symplectic manifolds, see Section \ref{gluingsection}. For instance, $\mathbb{T}^2\times\mathbb{S}^2$ is scattering-symplectic with singular hypersurface three torus $\mathbb{T}^3$. We also have that $\mathbb{S}^3\times\mathbb{S}^1$ is scattering-symplectic with singular hypersurface $\mathbb{S}^2\times\mathbb{S}^1$. While many scattering symplectic manifolds arise in this way, not all such structures can be obtained by gluing two strong convex fillings, see proposition \ref{notfill}.\\

\noindent {\bf \emph{ Poisson Cohomology.}}  Given a contact structure $\xi$ on $Z$, let $\alpha$ be a contact one form defining $\xi$. Consider the space of forms 

\centerline{$\Omega_\xi^{k}(Z):=\alpha \wedge \Omega^{k}(Z)\subset \Omega^{k+1}(Z)$.}

\begin{theorem} If $(M,\pi)$ is a scattering-Poisson manifold, with induced contact structure $\xi$ on $Z$, then the Poisson cohomology $H^p_{\pi}(M)$ is 

\centerline{$\displaystyle{H^{p}(M)\oplus H^{p-1}(Z)\oplus\Omega^{p-1}(Z)\oplus\Omega_\xi^{p-1}(Z)\oplus \ker(d\alpha\wedge:\Omega_\xi^{p-2}(Z)\to\Omega_\xi^{p}(Z)).}$}
\end{theorem}

\begin{remark} To provide this splitting of the Poisson cohomology into its various components, we fixed a tubular neighborhood of $Z$. A more invariant version can be found in Theorem \ref{scpoissoncohom}.\end{remark}

\noindent {\bf Acknowledgements. }

I would first and foremost like to thank Pierre Albin, who inspired this project and guided me through its completion. I am particularly grateful to Rui Loja Fernandes for his extensive comments.  I am grateful to Ioan M\u{a}rcu\c{t} and Ralph Klaasse for their suggestions and interesting questions. Travel support was provided by Pierre Albin's Simon's Foundation grant \#317883. Finally, I wish to thank the referee for the their comments and suggestions, which greatly improved the quality of the paper.

\section{Generalities}

In this section, we establish the framework we will use to understand scattering-symplectic structures. 

{\bf Rescaling Lie algebroids.} Richard Melrose explained how to rescale a vector bundle with a filtration over a hypersurface; see Proposition 8.1 in \cite{Melrose93}.
We adapt his construction to Lie algebroids. Note that Melrose rescaling is equivalent to lower elementary modification introduced by Travis Li in his thesis \cite{gualtieri2013symplectic,li2013constructions}. We will describe how to rescale a Lie algebroid with respect to a suitable subbundle over a hypersurface. Let $(M,Z)$ be a manifold $M$ with a coorientable hypersurface $Z\subset M$ and let $(\mathcal{A},\rho,[\cdot,\cdot])$ be a Lie algebroid over $M$. We will assume $Z$ to be compact, unless otherwise stated. Given a Lie subalgebroid $F\subseteq \mathcal{A}|_Z\to Z$, consider the space of sections:  

\begin{equation}\label{eq:D} \mathcal{D}=\left\{{u\in \mathcal{C}^{\infty}(M; \mathcal{A}): u\vert_Z\in\mathcal{C}^{\infty}(Z;F)}\right\}.\end{equation}

 \noindent $\mathcal{D}$ consists of the set of smooth sections of $\mathcal{A}$ that take values in $F$ at $Z$. There exists a Lie algebroid $(^F\mathcal{A}, \rho{}_F, [\cdot,\cdot]_F)$ over $M$ whose space of sections `is' $\mathcal{D}$ as defined in (\ref{eq:D}).

\begin{theorem}\label{scaling} There exists a vector bundle $^F\mathcal{A}\to M$ with a Lie algebroid morphism

\centerline{$i: {^F\mathcal{A}}\to \mathcal{A}$} \noindent that is an isomorphism over $M\setminus Z$ such that $i_*\mathcal{C}^\infty(M;\mathcal{A})=\mathcal{D}$. 
 \end{theorem} 
 
\begin{proof} To establish notation for verifying we have a Lie algebroid structure on our vector bundle we will include the steps to show that ${}^\mathcal{F}\mathcal{A}$ is a vector bundle. However, as mentioned above, these steps are solely due to Melrose. We begin by noting that $\mathcal{D}$ is preserved under multiplication by any smooth function on $M$ and thus is a $\mathcal{C}^{\infty}(M)$ module. For any $p\in M$, we will consider the ideal 

\centerline{$\mathcal{I}_p:=\left\{f\in\mathcal{C}^{\infty}(M):f(p)=0\right\}$.} \noindent Let 

\centerline{$\displaystyle{{}^F\mathcal{A}_p=\mathcal{D}/(\mathcal{I}_p\cdot\mathcal{D}) ~~\text{  and  }~~ ^{F}\mathcal{A}=\bigsqcup_{p\in M}{}^{F}\mathcal{A}_p.}$} \noindent Then because we can describe $\mathcal{A}_p$ as $\mathcal{C}^{\infty}(M; \mathcal{A})/(\mathcal{I}_p\cdot \mathcal{C}^{\infty}(M; \mathcal{A}))$, there exists a natural map $i: {^{F}\mathcal{A}_p}\to \mathcal{A}_p$ taking a section in $\mathcal{D}$ and evaluating it at the point $p$. For all $p\in M\setminus Z$, this map is an isomorphism. 

Suppose $p\in Z$ and let $\left\{{v_1,\dots,v_k}\right\}$ be a local basis of smooth sections of $F$. We can smoothly extend this to a basis  $\left\{{v_1,\dots,v_k,v_{k+1},\dots,v_M}\right\}$ of $\mathcal{A}$. Given a local $Z$ defining function $x$, any element $X\in\mathcal{D}$ is locally of the form 

\centerline{$\displaystyle{X=\sum_{i=1}^{k}g_iv_i+\sum_{j=k+1}^M xg_jv_j}$} \noindent for smooth functions $g_i$. Thus  $\left\{{v_1,\dots,v_k,xv_{k+1},\dots,xv_M}\right\}$ is a local basis of $^F\mathcal{A}$ and the coefficients $g_1,\dots,g_M$ give a local trivialization. 

Given any other local basis  $\left\{{v^\prime_1,\dots,v^\prime_k}\right\}$ for $F$, we can extend smoothly to a local basis $\left\{{v^\prime_1,\dots,v^\prime_k,v^\prime_{k+1},\dots,v^\prime_M}\right\}$ of $\mathcal{A}$. Each of $\left\{{v^\prime_1,\dots,v^\prime_k,xv^\prime_{k+1},\dots,xv^\prime_M}\right\}$ can be expressed as a smooth linear combination of $\left\{{v_1,\dots,v_k,xv_{k+1},\dots,xv_M}\right\}$. So this induces smooth transformations among the coefficients $g_i$ and $^{F}\mathcal{A}$ inherits a natural smooth bundle structure from $\mathcal{A}$ with bundle map $i:{^F\mathcal{A}}\to \mathcal{A}$. 

This ends the proof of Melrose's vector bundle construction. At this point we will explain how to equip the bundle with a Lie algebroid structure. Given $(\mathcal{A},[\cdot,\cdot],\rho)$, we can consider the rescaled vector bundle $^F\mathcal{A}$ as a Lie algebroid by taking anchor map $\rho_F$ to be the composition of the bundle map $i: {^F\mathcal{A}}\to \mathcal{A}$ with bundle map $\rho:\mathcal{A}\to TM$ and by taking as a Lie bracket $[\cdot,\cdot]_{F}=i^{-1}([i(\cdot),i(\cdot)])$, the bracket induced from the bracket $[\cdot,\cdot]$ on $\mathcal{A}$. 

We must verify that $[\cdot,\cdot]_{F}$ is a well-defined map $\Gamma({^F\mathcal{A}})\times\Gamma({^F\mathcal{A}}) \to\Gamma({^F\mathcal{A}})$, where $\Gamma({^F\mathcal{A}})$ is the $\mathcal{C}^\infty(M)$-module of smooth sections of ${^F\mathcal{A}}\to M$. As described above, given a local $Z$ defining function $x$, any elements $X,Y\in\mathcal{D}$ are locally of the form 

\centerline{$\displaystyle{X=\sum_{i=1}^{k}g_iv_i+\sum_{j=k+1}^M xg_jv_j~~\text{ and }~~ Y=\sum_{i=1}^{k}f_iv_i+\sum_{j=k+1}^M xf_jv_j.}$} \noindent Then  

\centerline{\scalebox{.95}{$\displaystyle{[X,Y]_F=\sum_{i,j=1}^{k}[g_iv_i,f_jv_j]+\sum_{i=1}^{k}\sum_{j=k+1}^{M}[g_iv_i,xf_jv_j]+\sum_{i=1}^{k}\sum_{j=k+1}^M [xg_iv_i,f_jv_j]+\sum_{i,j=k+1}^{M}[xg_iv_i,xf_jv_j].}$}} 
\noindent It suffices to show that each term of the sum restricts to a section of $\Gamma(F)$ or vanishes at $Z$.  
For $1 \leq i,j\leq k$, we have that $[g_iv_i,f_jv_j]|_Z\in\Gamma(F)$ by assumption and consequently $[g_iv_i,f_jv_j]\in\Gamma(^F\mathcal{A})$. Consider $[g_iv_i,xf_jv_j]=g_i\rho(v_i)(x)\cdot f_jv_j + x[g_iv_i,f_jv_j]$. Since we assumed that $\rho(F)|_Z\subseteq TZ$, $\rho(v_i)(x)=0$ which implies $[g_iv_i,xf_jv_j]\in\Gamma(^F\mathcal{A})$. By antisymmetry of the Lie bracket, $[xg_iv_i,f_jv_j]\in\Gamma(^F\mathcal{A})$. Finally we consider the term $[xg_iv_i,xf_jv_j]=xg_i\rho(v_i)(x)\cdot f_jv_j+x[xg_iv_i,f_jv_j]\in\Gamma(^F\mathcal{A})$. 

Our final task is to check that $[\cdot,\cdot]_F$ satisfies the Leibniz rule: 

\centerline{$[X,fY]_F=\rho_F(X)f\cdot Y + f[X,Y]_F$}

\noindent where $X,Y\in\Gamma(^F\mathcal{A})$, $f\in C^\infty(M)$ and $\rho_F(X)f$ is the Lie derivative of $f$ with respect to the vector field $\rho_F(X)$. Consider $[X,fY]_F=i^{-1}([i(X),i(fY)])=$ 

\centerline{$\displaystyle{i^{-1}(\rho(i(X))f\cdot i(Y)+f[i(X),i(Y)])=\rho_F(X)f\cdot Y+ f[X,Y]_F}$} \noindent because $i$ is an injective bundle map on $M\setminus Z$ and this identity extends by continuity. 
\end{proof} 

Note that the restriction of ${}^F\mathcal{A}$ to $Z$ is not $F$, but rather is a vector bundle of the same rank as $\mathcal{A}$. As explained by Melrose in Lemma 8.5 of \cite{Melrose93}, ${}^F\mathcal{A}|_Z$ is isomorphic to the graded bundle 
\begin{equation}\label{eq:Arestriction} F\oplus (N^*Z\otimes \mathcal{A}|_Z/F).\end{equation} 
The conormal bundle here makes this bundle invariant of choice of local $Z$ defining function. 
Further, $^F\mathcal{A}$ is, non-canonically, isomorphic to $\mathcal{A}$: Given a local $Z$ defining function $x$, we can map a local expression of an element of $\Gamma(^F\mathcal{A})$ to an element in  $\Gamma(\mathcal{A})$ by 

\centerline{$\displaystyle{\sum_{i=1}^{k}g_iv_i+\sum_{j=k+1}^M xg_jv_j\to \sum_{i=1}^{k}g_iv_i+\dfrac{1}{x}\sum_{j=k+1}^M xg_jv_j.}$} \vspace{2ex}

Next, we will explore some specific applications of this construction and introduce the scattering tangent bundle. \\

\noindent {\bf \emph{ Zero tangent bundle.}}\label{zerobundle} We can consider the tangent bundle of a manifold $TM\to M$ as a Lie algebroid with Lie bracket the standard bracket on vector fields and with anchor map the identity map. We can apply Theorem \ref{scaling} to a pair $(M,Z)$ and rescale $TM$ using the subbundle $0\to Z$.   The rescaled bundle ${^0TM}\to M$ is called the {\emph{zero tangent bundle}} and was introduced by Rafe Mazzeo and Richard Melrose in the context of manifolds with boundary \cite{mazzeo1988hodge,mazzeo1987meromorphic}. In this case $\mathcal{D}$ is the set of vector fields that vanish at $Z$. If $x,y_1,\dots,y_n$ are local coordinates near a point in $Z$, and $x$ is a local defining function for $Z$, then vector fields 

\centerline{$x\dfrac{\partial}{\partial x}, x\dfrac{\partial}{\partial y_1},\dots,x\dfrac{\partial}{\partial y_n}$} \noindent form a local basis for ${^0TM}$. Note that these do not vanish at $Z$ as sections of $^0TM$. The dual bundle ${^0T^*M}\to M$ is called the {\emph{zero cotangent bundle}} and is locally generated by 

\centerline{$\dfrac{dx}{x}, \dfrac{dy_1}{x},\dots,\dfrac{dy_n}{x}.$}
\noindent The anchor map of the $^0TM$ algebroid is evaluation in the tangent bundle and the bracket is induced by the standard Lie bracket on $TM$. \\

\noindent {\bf \emph{b-tangent bundle.}} We recover the \emph{$b$-tangent bundle} as formulated in \cite{guillemin2014symplectic}, by applying Theorem \ref{scaling} to the tangent bundle $TM$ over a pair $(M,Z)$, and taking as subbundle $TZ\to Z$. The $b$-tangent bundle is the vector bundle whose space of sections is 

\centerline{$\mathcal{D}=\left\{{u\in \mathcal{C}^{\infty}(M; TM):i\circ u\vert_Z\in \mathcal{C}^{\infty}(Z; TZ)}\right\}$,} \noindent the vector fields that are tangent to $Z$. If $x,y_1,\dots,y_n$ are local coordinates near a point in $Z$, and $x$ is a local  defining function for $Z$, then the vector fields and co-vectors
\begin{center}\begin{tabular}{l c c c c r}
$x\dfrac{\partial}{\partial x}, \dfrac{\partial}{\partial y_1},\dots,\dfrac{\partial}{\partial y_n}$& & & & &$\dfrac{dx}{x}, dy_1,\dots,dy_n$\\
\end{tabular}\end{center}
respectively form local bases for ${^{b}TM}$ and $^{b}T^*M$. The anchor map of the $^bTM$ algebroid is evaluation in the tangent bundle and the bracket is induced by the standard Lie bracket on $TM$. \\

\noindent {\bf \emph{ Scattering tangent bundle.}} Next we rescale the $b$-tangent bundle $^bTM \to M$ over a pair $(M,Z)$ using the subbundle $0\to Z$. The resulting bundle ${^{sc}TM}\to M$ is called the {\emph{scattering tangent bundle}}. If $x$ is a local defining function for $Z$, and $y_1,\dots,y_n$ are local coordinates in $Z$, then the vector fields 

\centerline{$x^2\dfrac{\partial}{\partial x}, x\dfrac{\partial}{\partial y_1},\dots,x\dfrac{\partial}{\partial y_n}$} \noindent form a local basis for ${^{sc}TM}$.  The dual bundle ${^{sc}T^*M}\to M$ is called the {\emph{scattering cotangent bundle}} and is locally generated by 

\centerline{$\dfrac{dx}{x^2}, \dfrac{dy_1}{x},\dots,\dfrac{dy_n}{x}.$}
\noindent The anchor map is evaluation in the $b$-tangent bundle and then in the tangent bundle $TM$. In the same way, the bracket is induced by the standard Lie bracket on $TM$. \vspace{2ex}

The next example is a generalization of the scattering tangent bundle that is comparable to Geoffrey Scott's $b^k$ generalization of the $b$-tangent bundle \cite{scott2016geometry}. \\

\noindent {\bf \emph{ Scattering-$k$ tangent bundle.}} Consider the scattering tangent bundle $^{sc}TM$ associated to $(M,Z)$. The scattering-2 tangent bundle $^{sc^2}TM$ is the vector bundle whose space of sections is $\mathcal{D}=\left\{{u\in \mathcal{C}^{\infty}(M; ^{sc}TM):u\vert_Z=0}\right\}.$ In other words, $^{sc^2}TM$ is rescaling $^{sc}TM$ by the subbundle $0\to Z$. 

In this fashion, given the scattering-$(k-1)$ tangent bundle $^{sc^{(k-1)}}TM$ associated to $(M,Z)$, the {\emph{scattering-$k$ tangent bundle}} $^{sc^k}TM$ is the vector bundle whose space of sections is $\mathcal{D}=$ $\left\{{u\in \mathcal{C}^{\infty}(M; ^{sc^{(k-1)}}TM):u\vert_Z=0}\right\}.$ If $x$ is a local defining function for $Z$, and $y_1,\dots,y_n$ are local coordinates in $Z$, then the vector fields 

\centerline{$x^{k+1}\dfrac{\partial}{\partial x}, x^k\dfrac{\partial}{\partial y_1},\dots,x^k\dfrac{\partial}{\partial y_n}$} \noindent form a local basis for ${^{sc^k}TM}$.  The dual bundle ${^{sc^k}T^*M}\to M$ is called the {\emph{scattering-$k$ cotangent bundle}} and is locally generated by 

\centerline{$\displaystyle{\frac{dx}{x^{k+1}}, \frac{dy_1}{x^k},\dots,\frac{dy_n}{x^k}.}$} 
\vspace{2ex}

 {\bf Morphisms of rescaled tangent bundles} In 1996 Ryszard Nest and Boris Tsygan introduced the notion of $b$-symplectic structures \cite{nest1996formal} on Melrose's $b$-tangent bundle. To elaborate more on the properties of $b$-symplectic manifolds, Victor Guillemin, Eva Miranda, and Ana Rita Pires \cite{guillemin2014symplectic} define the $b$-category. We will use this category for all rescaled bundles that arise from a sequence of rescalings of $TM$.   
\begin{definition} \cite{guillemin2014symplectic} A $b$-manifold is a pair $(M, Z)$ consisting of an oriented manifold $M$
and an oriented hypersurface $Z\subset M$. A $b$-map is a map 

\centerline{$f:(M_1, Z_1) \to (M_2, Z_2)$,} \noindent (a map $f:M_1\to M_2$ such that $f(Z_1)\subseteq Z_2$) with   $f$ transverse to $Z_2$ such that $f^{-1}(Z_2) = Z_1$. The $b$-category is
the category whose objects are $b$-manifolds and morphisms are $b$-maps.\end{definition} 

\subsection{$\mathcal{A}$-multivector fields}

Given a Lie algebroid $(\mathcal{A},[\cdot,\cdot],\rho)$, we can build {\bf $\mathcal{A}$-multivector fields}. These are 

\centerline{$^{\mathcal{A}}\mathcal{V}^k(M)=\mathcal{C}^{\infty}(M;\wedge^k\mathcal{A})$} \noindent and are analogous to multi vectorfields constructed using $\mathcal{A}=TM$. We can define an $\mathcal{A}$-Schouten bracket for these sections with a slight modification of the standard formula: The bracket of two $\mathcal{A}$-multivector fields is given by

\centerline{${\displaystyle [a_{1}\cdots a_{m},b_{1}\cdots b_{n}]=\sum _{i,j}(-1)^{i+j}[a_{i},b_{j}]a_{1}\cdots a_{i-1}a_{i+1}\cdots a_{m}b_{1}\cdots b_{j-1}b_{j+1}\cdots b_{n}}$} 
\noindent for $\mathcal{A}$-vector fields $a_i, b_j$ and where $[a_{i},b_{j}]$ denotes the $\mathcal{A}$-bracket. The bracket with a function $f$ is 
${\displaystyle [f,a_{1}\cdots a_{m}]=-i_{df}(a_{1}\cdots a_{m})}$ 
for $\mathcal{A}$-vector fields $a_i$ where $i_{df}$ denotes the interior product. 

This bracket provides enough structure for us to define an $\mathcal{A}$-Poisson structure on $\mathcal{A}$. 

\begin{definition} Given a Lie algebroid $(\mathcal{A},[\cdot,\cdot],\rho)$, an $\mathcal{A}$-bivector $\pi\in{}^{\mathcal{A}}\mathcal{V}^2(M)$ is {\bf Poisson} if $[\pi,\pi]=0$ where $[\cdot,\cdot]$ is the $\mathcal{A}$-Schouten bracket. \end{definition}

\subsection{Lifting} Given a Poisson structure $(M,\pi)$, we want to find an algebroid $\mathcal{A}$ where we can `view' $\pi$ as non-degenerate. Rigorously, we do this through lifts. 

\begin{definition}
Given a Poisson manifold $(M,\pi)$ and algebroid $\mathcal{A}\to M$, we say that $\pi$ is {\bf $\mathcal{A}$-Poisson} if there exists a Poisson bivector $\pi_\mathcal{A}\in\mathcal{C}^\infty(M;\wedge^2\mathcal{A})$ such that  $\rho(\pi_\mathcal{A})=\pi$. The bivector $\pi_\mathcal{A}$ is called an {\bf $\mathcal{A}$-lift} of $\pi$. \end{definition} 

\subsection{Anchored Lie algebroids} Frequently an $\mathcal{A}$-lift is degenerate. Since our goal is to eventually find a Lie algebroid with a lift that is non-degenerate, we can continue trying to lift our bivector by `stacking' Lie algebroids. 

\begin{definition}\emph{(\cite{Klaasse17}, Section 2.6).} Let $(\mathcal{B},[\cdot,\cdot],\rho_{\mathcal{B}})$ be a Lie algebroid of a manifold $M$.  Another Lie algebroid $(\mathcal{A},[\cdot,\cdot],\rho_\mathcal{A})$ is {\bf $\mathcal{B}$-anchored} if there exists a Lie algebroid morphism  $\phi:\mathcal{A}\to\mathcal{B}$ over the same base covering the identity: 

\begin{center}$\begin{tikzcd}
\mathcal{A} \arrow[rd, "\rho_{\mathcal{A}}"'] \arrow[rr, "\phi"] &  & \mathcal{B} \arrow[ld, "\rho_{\mathcal{B}}"] \\
 & TM & 
\end{tikzcd}$\end{center} \noindent We call $\phi$ the {\bf factoring map}. 
\end{definition}

The fact that we can anchor a Lie algebroid with another will  be extremely useful in the computation of  Lie algebroid cohomology. Recall that to every algebroid there is an associated cohomology theory. Recall, for any Lie algebroid $(\mathcal{A},[\cdot,\cdot]_\mathcal{A},\rho_\mathcal{A})$ over $M$, with dual $\mathcal{A}^*$, the degree $k$ $\mathcal{A}$-forms are 

\centerline{${}^\mathcal{A}\Omega^{k}(M)=\Gamma(\wedge^k\mathcal{A}^*),$} \noindent the sections of the $k$th exterior power of the dual bundle $\mathcal{A}^*$. The differential operator $d_\mathcal{A}$ acting on $^\mathcal{A}\Omega^{*}(M)$,
$d_\mathcal{A}:{^\mathcal{A}\Omega^{k}(M)}\to {^\mathcal{A}\Omega^{k+1}(M)}$ 
is defined by

\centerline{$\displaystyle{(d_\mathcal{A}\beta)(\alpha_0, \alpha_1,\dots, \alpha_k)= \sum_{i=0}^{k}(-1)^i\rho_{\mathcal{A}}({\alpha_i})\cdot\beta(\alpha_0,\dots,\hat{\alpha}_i,\dots,\alpha_k)}$}

\centerline{$\displaystyle{+\sum_{0\leq i<j\leq k}(-1)^{i+j}\beta([\alpha_i,\alpha_j]_\mathcal{A},\alpha_0,\dots,\hat{\alpha_i},\dots,\hat{\alpha_j},\dots,\alpha_k)}$}

\noindent for $\beta\in {^\mathcal{A}\Omega^{k}(M)}$, and  $\alpha_0,\dots, \alpha_k \in \Gamma(\mathcal{A}).$ This is a complex whose cohomology is called the Lie algebroid cohomology of $\mathcal{A}$ or $\mathcal{A}$-cohomology. 

We are now capable of generalizing the notion of a symplectic structure to any Lie algebroid. 

\begin{definition} Given a rank $2k$ Lie algebroid $(\mathcal{A},[\cdot,\cdot],\rho)$ over a manifold $M$, an $\mathcal{A}$-symplectic structure is a degree 2 $\mathcal{A}$-form $\omega\in{}^{\mathcal{A}}\Omega^2(M)$ such that $d_\mathcal{A}\omega = 0$ and $\wedge^k \omega$ is a nowhere vanishing section of ${}^{\mathcal{A}}\Omega^{2k}(M)$. 
\end{definition} 

When $\mathcal{A}=TM$ with the standard Lie bracket, we recover the definition of a symplectic form.

 \subsubsection{{\bf Local $Z$ defining functions and density bundles}}

Given a smooth manifold with cooriented hypersurface $Z\subset M$, the Lie algebroids $\mathcal{A}$ obtained from rescaling the tangent bundle, e.g. ${}^bTM,{}^0TM, {}^{sc}TM$, do not depend on a choice of local $Z$ defining function, hence neither do their $\mathcal{A}$-de Rham cohomologies. However, it is convenient for computations to work with a fixed defining function in a tubular neighborhood of $Z$.  Consequently, when we compute cohomology using a particular choice of local defining function we must conclude our computation by taking into account what happens to classes under change of local defining function. In the next example we will demonstrate the consequences of changing local defining function for some scattering forms. 






 \begin{example}Let $x$ be a $Z$ defining function on a manifold $(M,Z)$ where $Z$ separates $M$. Consider the scattering-form   $v=\dfrac{dx}{x^{k+1}}\wedge\alpha+\dfrac{\beta}{x^k}$ for $\alpha,\beta\in\Omega^*(M)$. We express any other $Z$ defining function $\widetilde{x}$ as $x=f\widetilde{x}$ for some nowhere vanishing function $f\in \mathcal{C}^{\infty}(M)$. Then
  
  \centerline{$\dfrac{dx}{x^{k+1}}\wedge\alpha+\dfrac{\beta}{x^k}=\dfrac{d\widetilde{x}}{\widetilde{x}^{k+1}}\wedge\left(\dfrac{1}{f^k}+\dfrac{\partial_{\widetilde{x}}  f}{f^{k+1}}\widetilde{x}\right)\alpha+\dfrac{1}{\widetilde{x}^k}\left(\dfrac{d_Zf}{f^{k+1}}\wedge\alpha+\dfrac{\beta}{f^k}\right).$} \noindent
 Thus the $\alpha$ and $\beta$ decomposition is highly dependent on the choice of $Z$ defining function. We will show in the course of proving Theorem \ref{scderham} that the only real ambiguity above is the scaling of $\alpha$ by $\dfrac{1}{f^k}$. This apparent dependence on $x$ is accounted for by the density bundle in equation (\ref{eq:Arestriction}), the restriction of a rescaled bundle to $Z$. A more complete discussion of densities can be found in section 4.5 of \cite{Melrose93}. A $Z$ defining function $x$ induces a trivialization of $(N^*Z)^s\to Z$. If $\widetilde{x}=fx$ is another $Z$ defining function, then $(d\widetilde{x})^s=f^s(dx)^s$ as sections of $(N^*Z)^s$. Because $\displaystyle{{}^{sc}TM|_Z\simeq \left(TZ\oplus (N^*Z\otimes TM|_Z/TZ)\right)\otimes N^*Z},$ a density bundle  appears in the cohomology groups to account for changes of $Z$ defining function and our presentation of the cohomology is independent of $Z$ defining function. 
 \end{example}


 

\subsection{Techniques for computation} Once we have identified a Lie algebroid where we can lift a Poisson bivector of interest, we need an explicit description of its Lie algebroid cohomology. In this section we will briefly discuss the computational techniques that we have found useful. \\

\noindent {\bf \emph{ Split short exact sequences.}} Recall that  an $\mathcal{A}$-anchored Lie algebroid $(\mathcal{B},[\cdot,\cdot]_\mathcal{B},\rho_\mathcal{B})$,  admits a `factoring' of its anchor map through the anchor map of the Lie algebroid $(\mathcal{A},[\cdot,\cdot],\rho_{\mathcal{A}})$: there exists a Lie algebroid map $\phi:\mathcal{B}\to\mathcal{A}$ such that $\rho_\mathcal{B}=\rho_\mathcal{A}\circ\phi$. We will only be considering instances when the factoring map $\phi$ is an isomorphism away from a collection of smooth hypersurfaces $Z\subset M$. In these cases we compute the Lie algebroid cohomology of $\mathcal{B}$ by forming a short exact sequence of complexes:  \begin{equation}\label{eq:ses}0\to {}^{\mathcal{A}}\Omega^k(M)\to {}^{\mathcal{B}}\Omega^k(M)\to\mathscr{C}^k(\mathcal{B},\mathcal{A})\to 0 \end{equation} \noindent where $\mathscr{C}^k(\mathcal{B},\mathcal{A})$ is the quotient ${}^{\mathcal{B}}\Omega^k(M)/{}^{\mathcal{A}}\Omega^k(M)$. The quotient complex inherits a differential $d_{\mathscr{C}}$ induced from the differential $d_\mathcal{B}$ on ${}^{\mathcal{B}}\Omega^k(M)$. Indeed, if $P$ is the projection $P:{}^{\mathcal{B}}\Omega^k(M)\to\mathscr{C}^k(\mathcal{B},\mathcal{A})$, then $d_\mathscr{C}(\eta)=P(d_\mathcal{B}(\theta))$ where $\theta\in{}^\mathcal{B}\Omega^k(M)$ is any form such that $P(\theta)=\eta$.

This short exact sequence allows us to break the computation of cohomology into a series of tractable bite-sized pieces.   In particular, the short exact sequence induces a long exact sequence in cohomology: 

\centerline{$\dots \xrightarrow{\partial} {}^{\mathcal{A}}H^n(M) \to {}^{\mathcal{B}}H^n(M) \to H^n(\mathscr{C})\xrightarrow{\partial} {}^{\mathcal{A}}H^{n+1}(M)\to\dots $} \noindent When the short exact sequence above is split, the boundary map $\partial$ is trivial and we have that 

\centerline{${}^{\mathcal{B}}H^k(M)\simeq {}^\mathcal{A}H^k(M)\oplus H^k(\mathscr{C})$.} 

Any Lie algebroid is trivially $TM$-anchored. Consider the $0$-tangent bundle. The anchor map is evaluation in $T M$. The sequence in equation (\ref{eq:ses}) becomes 

\centerline{$0\to \Omega^k(M)\to {}^{0}\Omega^k(M)\to\mathscr{C}^k\to 0$.} \noindent This sequence happens to be split. However, we typically can `do better' in our choice of split short exact sequence. Given the $b$-tangent bundle of a hypersurface $Z$ in manifold $M$, the Mazzeo-Melrose theorem tells us that the $b$-cohomology ${}^bH^p(M)\simeq H^p(M)\oplus H^{p-1}(Z)$. Since we already understand $b$-cohomology, we would be left with an easier computation if we anchored the 0-tangent bundle with the $b$-tangent bundle instead of the tangent bundle. Since the anchor maps for both of these algebroids are evaluation, the 0-tangent bundle is $b$-anchored and we have a short exact sequence 

\centerline{$0\to {}^b\Omega^k(M)\to {}^{0}\Omega^k(M)\to\mathscr{C}^k\to 0$.} 

\noindent {\bf \emph{ Taylor series expansion at a hypersurface.}} Another powerful tool in computing Lie algebroid cohomology is a Taylor series expansion. The idea is to consider an $\mathcal{A}$-anchored Lie algebroid $(\mathcal{B},[\cdot,\cdot]_\mathcal{B},\rho_\mathcal{B})$ where the factoring map $\phi:\mathcal{B}\to\mathcal{A}$ is an isomoprhism outside of a hypersurface (or collection of hypersurfaces) $Z\subset M$.  In a tubular neighborhood $U\supset Z$, we express a $\mathcal{B}$-form $\mu$ as $\mu = \eta + \theta$ where  $\eta \in\mathscr{C}^k$ and $\theta\in {}^\mathcal{A}\Omega^k(M)$. We say that $\eta$ is the singular part of $\mu$ and $\theta$ is the regular part. 

We can always construct a Taylor series for Lie algebroids obtained by a sequence of Melrose rescalings of $TM$. Recall that Taylor's theorem tells us for any integer $k\geq 1$ and any real smooth function $f:\mathbb{R}\to \mathbb{R}$, there exists a function $h_k:\mathbb{R}\to\mathbb{R}$ such that $$f(x)=f(0)+\dfrac{\partial f}{\partial x}(0)x +\dfrac{\partial^2 f}{\partial x^2} (0)\dfrac{x^2}{2!}+\dots + \dfrac{\partial^k f}{\partial x^k} (0)\dfrac{x^k}{k!}  +h_k(x)x^k$$ and $\displaystyle{\lim_{x\to 0}h_k(x)=0}$. In the setting of Lie algebroid forms, while we have coefficient functions that are  multivariate (i.e. locally $f:\mathbb{R}^{2n}\to \mathbb{R}$), we will only consider Taylor series expansions in the $x$-direction: Let $z=(z_1,\dots,z_n)$ denote a set of local coordinates for the hypersurface $Z=\left\{x=0\right\}$.  If we have a smooth function $f(x,z)$ and an integer $k\geq 1$,  $$f(x,z)=f(0,z)+\dfrac{\partial f}{\partial x}(0,z)x +\dfrac{\partial^2 f}{\partial x^2} (0,z)\dfrac{x^2}{2!}    +\dots +\dfrac{\partial^k f}{\partial x^k} (0,z)\dfrac{x^k}{k!} +h_k(x,z)x^k$$ and $\displaystyle{\lim_{x\to a}h_k(x)=0}$. Consequently, using a smooth partition of unity, we can provide Taylor series expansions in $x$ at the hypersurface $Z$ for forms of rescaled tangent bundles. To see an example of a Taylor series of a rescaled-tangent-bundle form, see for instance Theorem \ref{scderham}.

\subsection{Nondegenerate lifts \& $\mathcal{A}$-symplectic forms}  

\begin{definition} Given a Poisson manifold $(M,\pi)$ and a Lie algebroid $(\mathcal{A},[\cdot,\cdot],\rho)$ over $M$, we call $\pi$ an {\bf $\mathcal{A}$-symplectic structure} if there exists a non-degenerate Poisson $\mathcal{A}$-lift $\pi_{\mathcal{A}}$.
\end{definition} 

Given a non-degenerate Poisson structure $\pi_{\mathcal{A}}$, the sharp map $\pi_\mathcal{A}^\sharp:\mathcal{A}^\ast\to\mathcal{A}$ is an isomorphism. We denote the inverse map $(\pi_\mathcal{A}^\sharp)^{-1}: \mathcal{A}\to\mathcal{A}^\ast$ by $\omega^\flat_\mathcal{A}$. The associated two form $\omega_\mathcal{A}\in{}^{\mathcal{A}}\Omega^2(M)$ is an $\mathcal{A}$-symplectic form.

 \begin{remark} Given an $\mathcal{A}$-symplectic structure $(M,\pi)$, the condition that the $\mathcal{A}$-lift $\pi_\mathcal{A}$ is Poisson on the Lie algebroid $\mathcal{A}$, i.e. $[\pi_{\mathcal{A}},\pi_{\mathcal{A}}]=0$, is equivalent to the corresponding $\mathcal{A}$-symplectic form $\omega_\mathcal{A}$ being closed, i.e. $d\omega_\mathcal{A}=0$. This fact follows immediately from the definition of the exterior derivative $d$ for $\mathcal{A}$-forms, which uses the algebroid Lie bracket and anchor map in the exact same way as the classical exterior derivative on the tangent bundle. \end{remark} 

\subsection{Moser}   When we consider a Lie algebroid that is isomorphic to the tangent bundle away from a hypersurface (or collection of hypersurfaces), we can adapt the standard Moser technique of symplectic geometry to establish normal forms for an $\mathcal{A}$-symplectic form $\omega_{\mathcal{A}}$. 

\begin{lemma}{\bf ($\mathcal{A}$-Moser)}\label{AMoser} Let $(\mathcal{A},[\cdot,\cdot]_\mathcal{A},\rho_\mathcal{A})$ be a Lie algebroid over $M$ with $Z\subset M$ a hypersurface and assume that that $\rho_\mathcal{A}$ is an isomorphism on $M\setminus Z$. Let $\omega_1, \omega_2\in {^\mathcal{A}\Omega^2(M)}$ be $\mathcal{A}$-symplectic forms  such that  \begin{itemize} \item $[\omega_1] = [\omega_2]\in {}^{\mathcal{A}}H^2(U)$ for some tubular neighborhood $U$ of $Z$, and \item $\omega_1|_Z = \omega_2|_Z$

\end{itemize}  

\noindent Then there exists a neighbourhood $U^\prime$ of $Z$  and a Lie algebroid isomorphism $\phi: (\mathcal{A}|_{U^\prime} , \omega_1)\to (\mathcal{A}|_U, \omega_2)$ such that $\phi^\ast\omega_1 = \omega_2$. \end{lemma} 

\begin{proof} Let $U$ be a tubular neighborhood of $Z$ such that $[\omega_1]=[\omega_2]\in {}^{\mathcal{A}}H^2(U)$. Consider the convex combination of $\omega_1$ and $\omega_2$: 

\centerline{$\omega_t=\omega_1+t(\omega_2-\omega_1)$.}

\noindent The forms $\omega_1$ and $\omega_2$ representing the same class in $\mathcal{A}$-cohomology is insufficient for $\omega_t$ to be $\mathcal{A}$-symplectic. However, since $\omega_1|_Z=\omega_2|_Z$, there exists some neighborhood $U\supset U_0\supset  Z$ such that $\omega_t$ is non-degenerate for all $t\in [0,1]$.  Note that  $\omega_2-\omega_1$ is exact as an $\mathcal{A}$-form on $U_0$ because $[\omega_1]=[\omega_2]\in {}^{\mathcal{A}}H^2(U)$. Thus, 

\centerline{$\dfrac{d}{dt}\omega_t=\omega_2-\omega_1=d\sigma$} \noindent for some degree 1 $\mathcal{A}$-form $\sigma$. Following the standard Moser argument (Sec. 7.3 \cite{da2001lectures}), we want to solve the equation  

\centerline{$\mathcal{L}_{v_t}\omega_t +
\dfrac{d}{dt}\omega_t=di_{v_t}\omega_t +
d\sigma=0$.} \noindent Note that the non-degeneracy of $\omega_t$ implies the existence of a smooth family of sections $v_t \in \Gamma(\mathcal{A})$ satisfying 

\centerline{$i_{v_t}\omega_t + \sigma= 0$.}  \noindent

Next, we integrate $v_t$ to a family $\phi_t$ of Lie-algebroid automorphism of $\mathcal{A}$.  
See the appendix of \cite{crainic2003integrability} for information regarding integrating time-dependent sections of a Lie-algebroid. 
\end{proof}

\subsection{Almost regular Poisson structures} \label{AlmReg}

Scattering-symplectic structures are an example of almost regular Poisson structures in the sense of  Iakovos Androulidakis and Marco Zambon  \cite{AndroulidakisZambon17}. These are a large class of structures that includes $b$-symplectic (log-symplectic) \cite{guillemin2014symplectic} and $b^k$-symplectic \cite{scott2016geometry}. In this section we will briefly discuss some of the features of almost regular Poisson structures that will assist us in our study of scattering Poisson structures in particular. \begin{definition}\label{ARdef}(Theorem B, \cite{AndroulidakisZambon17}.) A Poisson manifold $(M, \pi)$ is {\bf almost regular} if  \begin{enumerate}\item the subset $M_{reg}$ where $\pi^\sharp$ has maximal rank is dense in $M$, and \item there is a unique integrable distribution $D$ of $M$, such that at every $x\in M_{reg}$ the symplectic
leaf of $(M,\pi)$ at $x$ is an open subset of the leaf of $D$ at $x$.\end{enumerate} \end{definition} Note that Androulidakis and Zambon provide an equivalent definition to the one provided here. However, as personal preference, we define almost regular Poisson structures via their characterization theorem. 

A sympectic manifold is trivially almost regular because $M_{reg}=M$. In this paper, we focus on Poisson structures that are non-degenerate away from a hypersurface $Z\subset M$. Thus $D=TM$, the entire tangent bundle of $M$. Accordingly, one could also say that scattering Poisson structures are `trivially' almost-regular. 

Our primary tool to study scattering Poisson structures will be Lie algebroids because they provide a context for us to view a Poisson structure as non-degenerate. Once we have a non-degenerate bivector, we have hope of adapting standard techniques of symplectic geometry to the Lie Algebroid setting. In particular, Lie algebroids whose anchor maps are an isomorphism almost everywhere will prove most useful for us. One class of Lie algebroids that are quite similar to the tangent bundle are called rescaled tangent bundles.

\subsection{Poisson cohomology \& rigged Lie algebroids} Recall (\cite{DufourZung05}, p. 39) for a general Poisson manifold $(M,\pi)$, the Poisson cohomology $H^*_\pi(M)$ is defined as the cohomology groups of the Lichnerowicz complex: This complex is formed using $\mathcal{V}^k(M):=\mathcal{C}^{\infty}(M;\wedge^k TM)$, smooth multivector fields on $M$.

\centerline{$\dots\to\mathcal{V}^{k-1}(M)\xrightarrow{d_{\pi}}\mathcal{V}^{k}(M)\xrightarrow{d_{\pi}}\mathcal{V}^{k+1}(M)\to\dots$} \noindent The differential 

\centerline{$d_{\pi}:  \mathcal{V}^k(M)\to\mathcal{V}^{k+1}(M)$} \noindent is defined as 

\centerline{$d_{\pi}=[\pi,\cdot],$} \noindent where $[\cdot,\cdot]$ is the Schouten bracket extending the standard Lie bracket on vector fields $\mathcal{V}^1(M)$.

We can easily generalize this cohomology to any Lie algebroid $(\mathcal{A},[\cdot,\cdot],\rho)$ over $M$. Given an $\mathcal{A}$-Poisson bivector $\pi_\mathcal{A}\in\mathcal{C}^\infty(M;\wedge^2 \mathcal{A})$, we define the $\mathcal{A}$-Poisson cohomology in the same way we define Poisson cohomology: The operator $d_{\pi_\mathcal{A}} = [\pi_\mathcal{A},\cdot]$, where $[\cdot,\cdot]$ is the $\mathcal{A}$-Schouten bracket, defines a differential on $\mathcal{A}$-multivector fields   $^{\mathcal{A}}\mathcal{V}^k(M)$. The {\bf $\mathcal{A}$-Poisson cohomology} of $(M,\pi_\mathcal{A})$, denoted ${}^{\mathcal{A}}H^*_{\pi}(M)$, is the cohomology of the {\bf $\mathcal{A}$-Lichnerowicz complex}:

\centerline{ $\dots\to{^{\mathcal{A}}\mathcal{V}^{k-1}(M)}\xrightarrow{d_{\pi_\mathcal{A}}}{^{\mathcal{A}}\mathcal{V}^{k}(M)}\xrightarrow{d_{\pi_\mathcal{A}}}{^{\mathcal{A}}\mathcal{V}^{k+1}(M)}\to\dots$}

\subsection{Modular vector fields} Once we have local normal forms for our Poisson structure, it is much easier to compute standard Poisson invariants. One of the easier examples is called the modular vector field. 

To compute the modular vector field we will use the description provided in section 2.6 of Dufour-Zung \cite{DufourZung05}. Recall that given a volume form $\Lambda\in\Omega^{2n}(M)$, we can define maps 

\centerline{$\Lambda^\flat:
\mathcal{V}^p(M)\to \Omega^{2n-p}(M) \text{ and }\Lambda^\sharp:\Omega^{2n-p}(M)\to \mathcal{V}^p(M)$}

 \noindent as follows: 

\centerline{$\Lambda^\flat(A)=i_A\Lambda\text{ and } \Lambda^\sharp(\eta)=-i_\eta\widehat{\Lambda}$} \noindent where $\hat{\Lambda}\in \mathcal{V}^{2n}(M)$ is the dual to $\Lambda$, i.e. $<\Lambda, \widehat{\Lambda}>=1$. Then the modular vector field $D_\pi\Lambda$ is given by $\Lambda^\sharp\circ d\circ \Lambda^\flat(\pi).$ 

Note that the modular vector field depends on the choice of volume form $\Lambda$. Consider another volume form $\widetilde{\Lambda}=f\Lambda$ for a positive function $f\in\mathcal{C}^\infty(M)$. Then 

\centerline{$D_{\pi}\widetilde{\Lambda} =D_{\pi}\Lambda -\pi^\sharp(\ln|f|)  $.}

But note that $f$ is non-zero, so $\ln|f|$ is a smooth function. Further, $-\pi^\sharp(\ln|f|)=d_\pi(\ln|f|)$, an exact element in $\mathcal{V}^1(M)$. Thus  the Poisson cohomology class $[D_{\pi}\Lambda]\in H_\pi^1(M)$ does not depend on $\Lambda$. 

\subsection{Computing Poisson cohomology} 

Because Poisson cohomology is quite challenging to compute, there are only very select cases where the answer is known. In the case of a symplectic manifold where the Poisson bivector is non-degenerate, the Poisson cohomology is isomorphic to the de Rham cohomology of the manifold $M$. The non-degeneracy of $\pi$ allows us to define an isomorphism $T^*M\to TM$ that induces this isomorphism in cohomology: $H^p(M)\simeq H^p_{\pi}(M)$. 

In the case when an $\mathcal{A}$-Poisson bivector $\pi_\mathcal{A}$ is non-degenerate, we recover an isomorphism of complexes analogous to the symplectic case. 

\begin{proposition} Let $(M,\omega_\mathcal{A})$ be an $\mathcal{A}$-symplectic manifold, and $\pi_\mathcal{A}$ the corresponding non-degenerate 
$\mathcal{A}$-Poisson structure. Then, the $\mathcal{A}$-Poisson cohomology $^{\mathcal{A}}H^*_\pi(M)$ is
isomorphic to the $\mathcal{A}$-cohomology $^{\mathcal{A}}H^*(M)$. 
\end{proposition} 

 As in the classic case when $\mathcal{A} = TM$, this proposition follows from the fact that we have an isomorphism of complexes induced by the map $\omega_\mathcal{A}^\flat:\mathcal{A}\to \mathcal{A}^\ast$. The full proof can be found in Theorem 30 of [\cite{guillemin2014symplectic}], simply replace every instance of the $b$-tangent bundle with Lie algebroid $\mathcal{A}$. 

Thus far we have constructed the following diagram of complexes. 

\centerline{$\xygraph{!{<0cm,0cm>;<1cm, 0cm>:<0cm,1cm>::}
!{(0,0)}*+{{(\mathcal{V}^*,d_\pi)}}="b"
!{(0,1.5)}*+{{({}^\mathcal{A}\mathcal{V}^*,d_{\pi_\mathcal{A}})}}="c"
!{(3,1.5)}*+{{({}^\mathcal{A}\Omega^*,d)}}="d"
"c":"b"^{\rho}_{}
"d":"c"_{\pi_\mathcal{A}^{\sharp}}^{\simeq} }$}

In very special cases, $\rho$ induces an isomorphism between $\mathcal{A}$-Poisson and Poisson cohomology. One such instance is the $b$-symplectic setting, a fact first shown by Ioan M\u{a}rcut and Boris Osorno Torres \cite{muarcuct2014deformations}. However, for general $\mathcal{A}$-Poisson structures, this will not be the case. In order to compute the Poisson cohomology of a generically symplectic Poisson manifold $(M,\pi)$, we will construct a complex $({}^\mathcal{R}\Omega^*,d)$, called the Rigged de Rham complex, that is isomorphic to the Lichnerowicz complex. 

\centerline{$\xygraph{!{<0cm,0cm>;<1cm, 0cm>:<0cm,1cm>::}
!{(0,0)}*+{{(\mathcal{V}^*,d_\pi)}}="b"
!{(0,1.5)}*+{{({}^\mathcal{A}\mathcal{V}^*,d_{\pi_\mathcal{A}})}}="c"
!{(3,1.5)}*+{{({}^\mathcal{A}\Omega^*,d)}}="d"
!{(3,0)}*+{{({}^\mathcal{R}\Omega^*,d)}}="e"
"c":"b"^{\rho}_{}
"d":"c"_{\pi_\mathcal{A}^{\sharp}}^{\simeq}
"e":"b"_{\pi^{\sharp}}^{\simeq}  }$} \noindent While computing cohomology using the Lichnerowicz complex can be quite intractable, understanding rigged Lie algebroid cohomology is typically easier.  

We begin by considering the $\mathcal{C}^\infty(M)$-submodule of $\mathcal{V}^1(M)$ generated by Hamiltonian vector fields, that is the $\mathcal{C}^\infty(M)$-module of
vector fields \begin{equation}\label{eq:rigged} \pi^\sharp(\Gamma(T^*M)).\end{equation}

Androulidakis and Zambon show that almost regular Poisson structures (Definition 2.5 \cite{AndroulidakisZambon17}) are precisely those whose module of Hamiltonian vector fields is projective, i.e. equation \ref{eq:rigged} defines the sections of a vector bundle. We will provide a sketch of their proof for the case when $(M,\pi)$ is generically symplectic. 

\begin{proposition}\label{propbundle} If $(M,\pi)$ is generically symplectic, then equation \ref{eq:rigged} defines the sections of a vector bundle $\mathcal{R}_\pi$. 
\end{proposition} 

\begin{proof} We will proceed using the Serre-Swan theorem: By arguing that $\pi^\sharp(\Gamma(T^\ast M))$ is a locally free projective module over $\mathcal{C}^\infty(M)$, we can conclude that  $\pi^\sharp(\Gamma(T^\ast M))$ is the space of sections of a vector bundle $\mathcal{R}_\pi$. First consider that $\pi^\sharp:\Gamma(T^\ast M) \to \Gamma(TM)$ is an injective map because $\pi$ is almost everywhere symplectic. Also note that $\pi^\sharp$ is a map of $\mathcal{C}^\infty(M)$-modules. Thus $\pi^\sharp(\Gamma(T^\ast M))$ is locally free and projective. \end{proof} 

\begin{definition} Given a generically symplectic Poisson manifold $(M,\pi)$, the {\bf rigged Lie algebroid $(\mathcal{R}_\pi, [\cdot,\cdot],\rho)$} of $\pi$ is the vector bundle $\mathcal{R}_\pi$ equipped with anchor map $\rho$ being evaluation in $TM$ and Lie bracket $[\cdot,\cdot]$ induced by the standard Lie bracket on $TM$.\end{definition} 

\begin{remark} The proof of Proposition \ref{propbundle} shows that if $\pi$ is generically symplectic, then $\mathcal{R}_\pi$ is canonically isomorphic to $T^\ast M$. In fact, as we will show shortly, the Lie algebroid structure that we have equipped $\mathcal{R}_\pi$ with makes  $(\mathcal{R}_\pi, [\cdot,\cdot],\rho)$  isomorphic to the standard Poisson Lie algebroid of the Poisson structure $\pi$. Accordingly, the face value of $(\mathcal{R}_\pi, [\cdot,\cdot],\rho)$ is not high. Instead, we value $\mathcal{R}_\pi$ because of the computations it enables us to complete. 

The name `rigged' is inspired from rigged Hilbert spaces, a construction designed to link distributions and square-integrable functions in spectral theory. In a similar fashion, the rigged Lie algebroid is a device that links the multivectors used to compute Poisson cohomology to singular differential forms. In particular, consider a Lie algebroid $\mathcal{A}$ with anchor map an almost everywhere isomorphism. Let $Z$ denote the singular locus of the anchor map. On $M\setminus Z$, the Lie algebroid forms ${}^\mathcal{A}\Omega^\ast (M)$ can be viewed as a subcomplex of $\Omega^\ast (M)$. In specific examples, we will characterize the types of singularities that the forms ${}^\mathcal{A}\Omega^\ast (M)$  exhibit along $Z$. This understanding allows us to calculate Poisson cohomology in a much easier fashion than wrestling with the Lichnerowicz complex directly.
\end{remark} 

\begin{lemma}\label{RiggedLemma} If $(M,\pi)$ is generically symplectic, then the Lie algebroid cohomology of the rigged Lie algebroid computes the Poisson cohomology of $\pi$. \end{lemma} 

\begin{proof} To build a map between the two complexes will require several steps. \vspace{2ex}

\noindent\begin{minipage}{.52\linewidth} Begin by considering (1):  since $\pi$ is generically symplectic, the set where $\pi^\sharp$ - viewed as a map between vector bundles - defines an isomorphism is dense. This induces (2), an injective map on sections. We obtain (3), a bijection, by  restricting to the codomain. (4) is the inverse map. Recall that the sheaf of sections $\Gamma(\mathcal{R})$ defines a vector bundle. Further, notice that $\pi^\sharp$ in step (2) is a $\mathcal{C}^\infty(M)$-linear map since it was induced by a bundle morphism. The operation of restriction still gives a $\mathcal{C}^\infty(M)$-linear map. 

\end{minipage}\hfill\begin{minipage}{.05\linewidth} ~

(1)\vspace{1ex}

(2)\vspace{1ex}

(3)\vspace{1ex}

(4)\vspace{1ex}

(5)\vspace{1ex}

(6)\vspace{1ex}

(7)

\end{minipage}\begin{minipage}{.38\linewidth}
{\singlespace
\begin{center}

\scalebox{.9}{$
\pi^\sharp:T^\ast M\to TM$}

\rotatebox{-90}{\hspace{-2ex}$\rightsquigarrow$}

\scalebox{.9}{$ \pi^\sharp:\Gamma(T^\ast M)\to \Gamma(TM)$}

\rotatebox{-90}{\hspace{-2ex}$ \rightsquigarrow $}

 \scalebox{.9}{$\pi_{\text{\tiny rest}}^\sharp:\Gamma(T^\ast M)\to  \pi^\sharp(\Gamma(T^\ast M))\simeq\Gamma(R)$}
 
\rotatebox{-90}{ \hspace{-2ex}$ \rightsquigarrow $}

\scalebox{.9}{$(\pi^\sharp_{\text{\tiny rest}})^{-1}:\Gamma(\mathcal{R})\to \Gamma(T^\ast M)$}

\rotatebox{-90}{\hspace{-2ex}$\rightsquigarrow $}

\scalebox{.9}{$(\pi^\sharp_{\text{\tiny rest}})^{-1}:\mathcal{R}\to T^\ast M$}

\rotatebox{-90}{\hspace{-2ex}$\rightsquigarrow $}

\scalebox{.9}{$\omega^\flat:TM\to\mathcal{R}^\ast$}

\rotatebox{-90}{\hspace{-2ex}$ \rightsquigarrow$}

\scalebox{.9}{$ \omega^\flat:\Gamma(TM)\to\Gamma(\mathcal{R}^\ast)$}

\end{center}}\end{minipage}\vspace{2ex}

\noindent Thus (4) is also a $\mathcal{C}^\infty(M)$-linear map and induces (5), an invertible bundle map. We dualize to obtain (6).  This induces (7), a map on sections. 

By taking exterior powers of the map $\omega^\flat$, we can extend it to a $\mathcal{C}^\infty(M)$-linear isomorphism
 
 \centerline{$\bar{\omega}:\mathcal{V}^p(M)\to {^{\mathcal{R}}\Omega^p(M)}.$}  \noindent This will be the desired isomorphism of complexes. 
 
Next we will show that for any smooth multivector field $\eta$ on a given generically symplectic Poisson manifold $(M,\pi)$, we have $\bar{\omega}(d_\pi(\eta))=-d(\bar{\omega}(\eta))$.  We will proceed by induction on the degree of $\eta$ and by using the Leibniz rule. Let $\eta$ be a degree $0$ form, that is $\eta\in\mathcal{C}^\infty(M)$. Then $\bar{\omega}(\eta)=\eta$ and  $-d(\bar{\omega}(\eta))=-d\eta$. Consider $d_\pi(\eta)= [\pi,\eta]=-X_\eta$, the Hamiltonian vector field of $\eta$. On $M_{reg}$, we have that $\pi$ defines a symplectic form $\omega$. Further, on $M_{reg}$, $\bar{\omega}$ is given by contraction with $\omega$. Thus on $M_{reg}$, we have that 

\centerline{$\bar{\omega}(-X_\eta) = -i_{X_\eta}\omega =-d\eta$,} \noindent by the correspondence of Hamiltonian vector fields between non-degenerate $\pi$ and corresponding symplectic form $\omega$. By continuity this identity extends to all of $M$. If $\eta =d_\pi f$ is an exact 1-vector field, then $\bar{\omega}(d_\pi(d_\pi f))=\bar{\omega}(0)=0$ and 

\centerline{$d(\bar{\omega}(d_\pi f)) = d(\bar{\omega}(X_f))=d(df)=0$.} \noindent By the Leibniz rule, the statement is true for all multivector fields. Thus we have shown that, up to a sign, the map $\bar{\omega}$ is an isomorphism that intertwines the differential $d$ of the rigged algebroid $\mathcal{R}$ de Rham complex with the differential operator $d_\pi$ of the Lichnerowicz complex. Hence $\bar{\omega}: H^p_\pi(M)\to {}^{\mathcal{R}}H^p(M)$ is an isomorphism.\end{proof}


\begin{remark} Equation \ref{eq:rigged} defines a vector bundle $\mathcal{R}_\pi$ for any $(M,\pi)$ almost regular Poisson. However, if the maximal rank of $\pi$ is less than the dimension of $M$, in general the Lie algebroid cohomology of $\mathcal{R}_\pi$ does not directly compute the Poisson cohomology of $\pi$. This is because $\mathcal{R}_\pi$ encodes information about the leaves of the regular foliation arising from distribution $D$ and does not contain `transverse' information.   \end{remark} 


\section{Scattering symplectic geometry}

In this section, we will explore in greater depth the properties of scattering-symplectic manifolds.

\subsection{Scattering symplectic forms}\label{scatteringsymplecticgeometry}

Given a manifold and hypersurface $(M,Z)$, a \emph{scattering symplectic structure} on $M$ is a closed, non-degenerate, degree-2 scattering-form $\omega\in \mathcal{C}^\infty(M;\wedge^2({}^{sc}T^\ast M))=:{}^{sc}\Omega^2(M)$.

Similar to the role that second degree de Rham cohomology $H^2(M)$ plays in determining the behavior of a symplectic form $\omega$, the second  degree scattering-cohomology  ${}^{sc}H^2(M)$ will play a key role in determining the behavior of a scattering-symplectic form $\omega$. Next we obtain an explicit description of this cohomology.

\subsubsection{Scattering de-Rham cohomology} 

Our computation of scattering de-Rham cohomology utilizes the following result of Rafe Mazzeo and Richard Melrose (\cite{Melrose93}, Prop. 2.49).  

\begin{theorem}[Mazzeo-Melrose] Let ${}^bTM$ be the b-tangent bundle associated to $(M,Z)$. The $b$-cohomology is 

\centerline{$^{b}H^p(M)\simeq H^p(M)\oplus H^{p-1}(Z).$} 
\end{theorem} 
 
\begin{theorem}\label{scderham} Let $(M,Z)$ be a manifold $M$ with co-oreintable hypersurface $Z\subset M$. Then $^{sc}H^p(M)$, the Lie algebroid cohomology of the scattering tangent bundle $^{sc}TM$ over $(M,Z)$, is isomorphic to

\centerline{$^{b}H^p(M)\oplus \Omega^{p-1}(Z;(N^*Z)^{-p})\simeq H^p(M)\oplus H^{p-1}(Z)\oplus \Omega^{p-1}(Z;(N^*Z)^{-p}).$} \end{theorem} 

\begin{proof}
The bundle map $i: {^{sc}TM}\to {^bTM}$ constructed in Theorem \ref{scaling} is given by evaluation and hence fits into a short exact sequence of complexes 

\centerline{$0\to {^b\Omega^{p}(M)}\xrightarrow{i^*}{^{sc}\Omega^{p}(M)}\xrightarrow{\pi} \mathscr{C}^p\to 0$}
\noindent where  

\centerline{$\mathscr{C}^p={}^{sc}\Omega^{p}(M)/{}^b\Omega^{p}(M).$} \noindent The differential on $^{\mathscr{C}}d$ is induced by the differential $^{sc}d$ on ${^{sc}\Omega^{p}(M)}$: if $\pi$ is the projection $^{sc}\Omega^{p}(M)\to{^{sc}\Omega^{p}(M)}/{^b\Omega^{p}(M)},$ then  $^{\mathscr{C}}d(\eta)=\pi({}^{sc}d(\theta))$ where  $\theta\in{}^{sc}\Omega^{p}(M)$ is any form such that $\pi(\theta)=\eta$. Hence $({}^\mathscr{C}d)^2=0$ and $(\mathscr{C}^*,{}^\mathscr{C}d)$ is in fact a complex. 

Given a tubular neighborhood $\tau=Z\times(-\varepsilon,\varepsilon)_x$ of $M$ near $Z$, note that $x$ defines a trivialization $t_x:N^*Z\to\mathbb{R}$ of $N^*Z$. We can write a degree $p$ scattering form $\nu\in {}^{sc}\Omega^p(M)$ as 

\centerline{$ \displaystyle{\nu= \theta + \sum_{i=0}^{p-1}\left( \frac{dx}{x^{p+1}}\wedge \alpha_ix^i + \dfrac{\beta_ix^i}{x^p}\right)}$} \noindent where $\theta\in{}^b\Omega^p(M)$, and $\alpha_i$, $\beta_i\in\Omega^{*}(Z)\simeq \Omega^{*}(Z;(N^*Z)^{-p})$ by $(t_x)_*$. 

We write $\mathcal{R}_b(\nu)=\theta$ and $\mathcal{S}_b(\nu)=\nu-\mathcal{R}_b(\nu)$ for `regular' and `singular' parts. It is easy to see that $\mathcal{R}_b({}^{sc}d\nu)={}^{sc}d(\mathcal{R}_b(\nu))$ and $\mathcal{S}_b({}^{sc}d\nu)={}^{sc}d(\mathcal{S}_b(\nu))$. Thus the trivialization $\tau$ induces a splitting ${}^{sc}\Omega^*(M)={}^b\Omega^*(M)\oplus\mathscr{C}^*$ as complexes. As a consequence we have ${}^{sc}H^p(M)={}^bH^p(M)\oplus  H^p(\mathscr{C}^{*})$ and are left to compute the cohomology of the quotient complex.  

Because 

\centerline{$\displaystyle{\pi({}^{sc}d\nu)= \frac{1}{x^{p+1}}\left(\sum_{i=0}^{p-1}\frac{dx}{x}\wedge\left(x^{i+1}(-d\alpha_i-(p-i)\beta_i)\right)+x^{i+1}d\beta_i\right)}$} \noindent then $\pi({}^{sc}d\nu)=0$ if and only if $\beta_i=\dfrac{-d\alpha_i}{(p-i)}$ for all $i=0,\dots,p-1$. Thus 

\centerline{$\displaystyle{\ker({}^\mathscr{C}d:\mathscr{C}^p\to\mathscr{C}^{p+1})=\left\{\sum_{i=0}^{p-1}\left( \frac{dx}{x^{p+1}}\wedge \alpha_ix^i - \dfrac{d\alpha_ix^i}{(p-i)x^p}\right)\bigg\rvert \alpha_i\in\Omega^{p-1}(Z)\right\}.}$}

Now we will consider the image of  ${}^{\mathscr{C}}d:\mathscr{C}^{p-1}\to\mathscr{C}^{p}$. There exists 

\noindent $\displaystyle{\widetilde{\nu}=\sum_{i=1}^{p-1}\frac{\alpha_ix^{i-1}}{-(p-i)x^{p-1}}\in\mathscr{C}^{p-1}}$ such that $\displaystyle{^{\mathscr{C}}d\widetilde{v}=\sum_{i=1}^{p-1}\left( \frac{dx}{x^{p+1}}\wedge \alpha_ix^i - \dfrac{d\alpha_ix^i}{(p-i)x^p}\right)}$, and hence 

\centerline{$\nu-d\widetilde{\nu}= \dfrac{dx}{x^{p+1}}\wedge \alpha_0 - \dfrac{d\alpha_0}{px^p}.$} \noindent Since 

\centerline{$\text{Im}({}^{\mathscr{C}}d)\subseteq\left\{{\alpha_0=0}\right\},$} \noindent this shows that each such form represents a distinct cohomology class. Thus we have identified $H^p(\mathscr{C}^*)$. 

Next, consider the effect of choice of local $Z$ defining function. We will show that we can identify 

\centerline{$H^p((\mathscr{C}^{*},{}^\mathscr{C}d))\simeq \Omega^{p-1}(Z;(N^*Z)^{-p}) \text{ by identifying } \dfrac{dx}{x^{p+1}}\wedge \alpha_0 - \dfrac{d\alpha_0}{px^p}\text{  with }\alpha_0.$} \noindent Indeed, each choice of $x$ defines a trivialization of $N^*Z$, $t_x:N^*Z\to\mathbb{R}$.  Then $(t_x)_*$ gives an isomorphism $\Omega^{p-1}(Z)\simeq \Omega^{p-1}(Z;(N^*Z)^{-p})$. To see that this is well-defined, note that changing $x$ to another local $Z$ defining function $\widetilde{x}$, means that $\widetilde{x}=\phi x$ for some non-vanishing function $\phi\in\mathcal{C}^\infty(M)$. Then $\dfrac{d\widetilde{x}}{\widetilde{x}}=\dfrac{dx}{x}+\dfrac{d\phi}{\phi}$. Since $\phi$ is a non-vanishing function,  

\centerline{$\left[\dfrac{d\widetilde{x}}{\widetilde{x}^{p+1}}\wedge\alpha_0\right]\text{ and }\left[\dfrac{dx}{x^{p+1}}\wedge\dfrac{\alpha_0}{\phi^p}\right]$} \noindent are representatives of the same cohomology class in $H^p(\mathscr{C})$ and 

\centerline{$(d\widetilde{x})^{-p}=\phi^{-p}(dx)^{-p}$} \noindent gives the change of trivialization of the density bundle $(N^*Z)^{-p}$. Hence, the cohomology group at $\mathscr{C}^p$ is $\Omega^{p-1}(Z;(N^*Z)^{-p})$, smooth $p-1$ forms on $Z$ valued in $(N^*Z)^{-p}$.  

We have shown that $^{sc}H^p(M)\simeq$ $^{b}H^p(M)\oplus \Omega^{p-1}(Z;(N^*Z)^{-p}).$ The final isomorphism is a consequence of $^{b}H^p(M)\simeq H^p(M)\oplus H^{p-1}(Z)$. \end{proof} 

{\bf Connection to contact geometry.} As we noted above, scattering-symplectic geometry includes the study of the standard Euclidean symplectic form at infinity. We also observed that the standard form on $\mathbb{R}^{2n}$ extends to a scattering-symplectic form that induces a contact structure on the boundary sphere at infinity. Contact structures arising in this way are imposed by all scattering-symplectic structures. The existence of a defining function $x$ in a tubular neighborhood of $Z$ means that $Z$ is co-orientable inside $M$. The scattering-symplectic structure imposes the further restriction that $Z$ is also  co-orientable as a contact manifold. 

\begin{proposition}\label{contactprop} If $M$ be a manifold with cooriented hypersurface $Z\subset M$ and scattering-symplectic form $\omega$, then $\omega$ induces a co-oriented contact structure $\xi$ on $Z$. 
\end{proposition} 

\begin{proof} Let $x$ be a local $Z$ defining function. A scattering 2-form $\omega$ can be expressed near $Z$ as 

\centerline{$\omega = \dfrac{dx}{x^3}\wedge\alpha+\dfrac{\beta}{x^2}$} \noindent for some smooth forms $\alpha\in\Omega^1(M)$ and $\beta\in\Omega^2(M)$. Because $\omega$ is closed,  $\beta = \dfrac{-d\alpha}{2}$ at $Z$. Further, because $\omega$ is a non-degenerate scattering 2-form, we know that 

\centerline{$\displaystyle{\frac{dx}{x^{2n+1}}\wedge\alpha\wedge \left(\dfrac{-d\alpha}{2}\right)^{n-1}\neq 0}$} \noindent at $Z$. Thus $\alpha\wedge (d\alpha)^{n-1}\neq 0$ as a smooth form on $Z$. If we express $\omega$ using a different local $Z$ defining function $\widetilde{x}$ such that $\phi\widetilde{x}=x$ for some non-vanishing function $\phi\in\mathcal{C}^\infty(M)$, then 

\centerline{$\omega|_Z= \dfrac{dx}{x^3}\wedge\alpha-\dfrac{d\alpha}{2x^2}=\dfrac{1}{(\phi\widetilde{x})^2}\left(\dfrac{d\phi}{\phi}+\dfrac{d\widetilde{x}}{\widetilde{x}}\right)\wedge\alpha -\dfrac{d\alpha}{2(\phi\widetilde{x})^2}= \dfrac{d\widetilde{x}}{\widetilde{x}^3}\wedge\left(\dfrac{\alpha}{\phi^2}\right)-\dfrac{d\left(\alpha/\phi^2\right)}{2\widetilde{x}^2}.$}  \noindent The contact form $\widetilde{\alpha}$ induced by $\omega$ expressed in $\widetilde{x}$ satisfies $\widetilde{\alpha}=\dfrac{\alpha}{\phi^2}$, and thus is conformally equivalent to $\alpha$. Thus the scattering symplectic form $\omega$ induces a conformal class of contact forms defining the contact structure $\ker\alpha=\xi$ on $Z$. 
\end{proof} 

We will explore the relationship between a contact hypersurface and a scattering-symplectic manifold in Section \ref{Contact}. The existence of an induced contact structure evidences the fact that scattering-symplectic structures are sufficiently rigid to all locally look the same. 

\begin{proposition}\label{scdarboux} (a sc-Darboux theorem) Let $\omega$ by a sc-symplectic form on $(M,Z)$ and let $p\in Z$. There exists a coordinate chart $(U,x_1,y_1,\dots,x_n,y_n)$ centered at $p$ such that on $U$, the hypersurface $Z$ is locally defined by $\left\{{x_1=0}\right\}$, and 

\centerline{$\displaystyle{\omega = \frac{dx_1}{x_1^3}\wedge \Big(dy_1+\sum_{i=2}^{n} y_idx_i-x_idy_i\Big)+\sum_{i=2}^n\frac{dx_i\wedge dy_i}{x_1^2}.}$}
\end{proposition} 

\begin{proof} Let $\omega$ be a scattering symplectic form on $(M,Z)$. By Proposition \ref{contactprop}, given a local $Z$ defining function $x$, 

\centerline{$\omega|_Z=\dfrac{dx}{x^3}\wedge\alpha-\dfrac{d\alpha}{2x^2}$} \noindent where $\alpha$ is a contact form on $Z$.  Let $p\in Z$. In a neighborhood $U_p\subset Z$, there exist contact-Darboux coordinates $y_1,x_2,y_2,\dots,x_n,y_n$ such that 

\centerline{$\displaystyle{\alpha= dy_1+\sum_{i=1}^n(y_idx_i- x_idy_i)}.$} \noindent Then 

\centerline{$\displaystyle{\omega|_Z = \frac{dx}{x^3}\wedge \Big(dy_1+\sum_{i=2}^{n} y_idx_i-x_idy_i\Big)+\sum_{i=2}^n\frac{dx_i\wedge dy_i}{x^2}.}$} \noindent In the set $U_p\times \left\{{|x|<\varepsilon}\right\}$, choose 

\centerline{$\displaystyle{\omega_0= \frac{dx}{x^3}\wedge \Big(dy_1+\sum_{i=2}^{n} y_idx_i-x_idy_i\Big)+\sum_{i=2}^n\frac{dx_i\wedge dy_i}{x^2}.}$}

\noindent  By Lemma \ref{AMoser}, we have the existence of a Lie algebroid isomorphism such that $\phi^\ast\omega = \omega_0$.  However, we will continue our analysis to prove the stronger statement that we can guarantee $\phi$ also covers the identity on $Z$. 

\noindent  Consider that $\omega - \omega_0 = \dfrac{dx}{x}\wedge \beta_1 + \beta_2$ for some closed $1$-form $\beta\in H^1(U_p)$ and a closed $2$-form $\beta_2\in H^2(U_p)$.  Since $\beta_1$ and $\beta_2$ are smooth differential forms on  contractible set $U_p$, we can appeal to the Poincar\'{e} lemma and there exists $f\in\Omega^0(M)$ and $\theta\in \Omega^1(M)$ so that $\sigma := f\dfrac{dx}{x} + \theta$ satisfies $d\sigma = \omega - \omega_0$.  As described in the proof of Lemma \ref{AMoser}, we want $v_t$ such that $i_{v_t}\omega_t=- f\dfrac{dx}{x} - \theta$. Consider that $$\omega_t = \dfrac{dx}{x^3}\wedge \alpha_t - \dfrac{d\alpha_t}{2x^2}$$ for a smooth family of differential forms $\alpha_t$. Because $\omega_t$ is so singular in every single one of its terms, in order for $v_t$ to contract in and create the less-singular $b$-form, we must have that $v_t$ has at least one copy of $x$ in front of every term. Thus $v_t|_Z=0$. Consequently, when we integrate $v_t$ into the automorphism $\phi_t$, this bundle map will cover the identity on $Z$. \end{proof} 

We can further employ Lemma \ref{AMoser} to establish a tubular neighborhood theorem for $\omega$ near $Z$.  

Recall from Theorem \ref{scderham} that 

\centerline{${}^{sc}H^2(M)\simeq H^2(M)\oplus H^1(Z)\oplus \Omega^{1}(Z;(N^*Z)^{-2}).$} \noindent Given a cohomology class $[\sigma]\in {}^{sc}H^2(M)$, we can associate to it a \emph{decomposition} $(a,b_1,b_2)$ where $a\in \Omega^{1}(Z;(N^*Z)^{-2})$, $b_1\in H^1(Z)$, and $b_2\in H^2(M)$. We will consider scattering symplectic forms $\omega$ and their cohomology decompositions $(a,b_1,b_2)$. A given local $Z$ defining function $x$ gives us a trivialization $t_x:N^*Z\to \underline{\mathbb{R}}$ and defines a smooth contact form $\alpha=(t_x)_*(a)\in\Omega^1(Z)$. We will show that for any $\beta_i\in b_i$, there is a tubular neighborhood of $Z$ such that  \begin{equation}\label{eq:normalform}\omega=\dfrac{dx}{x^3}\wedge (\alpha+x^2\beta_1)-\dfrac{d\alpha}{2x^2}+\beta_2.\end{equation} 

\begin{proposition}\label{sctube1} Let $(M,Z,\omega)$ be a scattering symplectic manifold. Given a local $Z$ defining function $x$, there exists a tubular neighborhood $U\supset Z$, a contact form $\alpha$, and closed forms $\beta_1\in\Omega^1(Z)$, $\beta_2\in\Omega^2(Z)$ such that on $U$ there exists a scattering-symplectomorphism pulling $\omega$ back to (\ref{eq:normalform}). Further, if $[\omega]\in {}^{sc}H^2(M)$ has decomposition $(a,b_1,b_2)$, then $\beta_1$ can be any $\beta_1\in b_1$ and $\beta_2$ can be any $\beta_2\in b_2$. 
\end{proposition} 

\begin{proof} Let $\omega$ be a scattering symplectic form on a manifold $(M,Z)$ with cohomology class decomposition $(a,b_1,b_2)$. Let $x$ be a local $Z$ defining function and $t_x:N^*Z\to\mathbb{R}$ the associated trivialization. Let $\alpha$ denote $(t_x)_*(a)$ and let $U$ be a tubular neighborhood of $Z$. Choose a closed form $\beta_2$ that is cohomologous to $b_2|_U$ and choose a representative $\beta_1\in b_1$. Let 

\centerline{$\omega_0=\dfrac{dx}{x^3}\wedge (\alpha+x^2\beta_1)-\dfrac{d\alpha}{2x^2}+\beta_2.$} \noindent   

We have satisfied the conditions of Lemma \ref{AMoser}, and thus we have $\phi$ such that $\phi^\ast \omega = \omega_0$. However, as explained in the proof of Proposition \ref{scdarboux}, the scattering tangent bundle has the special quality that the vector field $v_t$, which we integrate to construct $\phi$, additionally satisfies $v_t|_Z=0$. Thus $\phi$ covers the identity on $Z$. \end{proof} 

In the previous two results, we wrote a scattering-symplectic form in a standard way by assuming that $\omega$ fixed a contact structure on $Z$. In general a scattering-symplectomorphism will induce a contactomorphism rather than merely fix the contact structure.    
\begin{proposition} If there exists a scattering-symplectomorphism

 \centerline{$\Phi:(M_1,Z_1,\omega_1)\to(M_2,Z_2,\omega_2),$}\noindent  then $\Phi|_{Z_1}:Z_1\to Z_2$ is a contactomorphism between the contact structures induced by $\omega_1$ and $\omega_2$ respectively. 
\end{proposition} 

\begin{proof} Given a local $Z_2$ defining function $x_2$, we can write  

\centerline{$\displaystyle{\omega_2|_{Z_2}=\dfrac{dx_2}{x_2^3}\wedge\alpha_2-\dfrac{1}{2}\dfrac{d\alpha_2}{x_2^2}}$} \noindent for some contact form $\alpha_2$ on $Z_2$. Then 

\centerline{$\displaystyle{\Phi^*\omega_2|_{Z_2}=\frac{d(\Phi^*x_2)\wedge\Phi^*\alpha_2}{(\Phi^*x_2)^3}-\frac{1}{2}\frac{\Phi^*d\alpha_2}{(\Phi^*x_2)^2}=\frac{dx_1}{x_1^3}\wedge\alpha_1-\frac{1}{2}\frac{d\alpha_1}{x_1^2}=\omega_1|_{Z_1}}$} \noindent for some local $Z_1$ defining function $x_1$ and contact form $\alpha_1$. We will compare the terms in this equality. Note since $\Phi$ preserves the singular locus of $\omega_1$ and $\omega_2$, then $\Phi^*x_2=fx_1$ for non-vanishing $f\in\mathcal{C}^{\infty}(M_1)$. So 

\centerline{$\displaystyle{d(\Phi^*x_2)=d(fx_1)=\dfrac{x_1df+fdx_1}{x_1^3f^3}~\text{ and }~\dfrac{x_1df+fdx_1}{x_1^3f^3}\Big\rvert_Z=\dfrac{dx_1}{x_1^3f^2}.}$} \noindent Then $\dfrac{dx_1}{x_1^3f^2}\wedge\Phi^*\alpha_2=\dfrac{dx_1}{x_1^3}\wedge\alpha_1$ and thus $\Phi^*\alpha_2=f^2\alpha_1$. 
\end{proof}  

In fact, we can use certain contactomorphisms to construct local scattering-symplectomorphisms. 

\begin{proposition}\label{sctube2} Let $\omega$ and $\widetilde{\omega}$ be scattering-symplectic forms on $(M,Z)$ with cohomology decompositions $(a,b_1,b_2)$ for $[\omega]$ and $(\widetilde{a},\widetilde{b}_1,\widetilde{b}_2)$ for $[\widetilde{\omega}]$. Let $x$ be a local $Z$ defining function and consider the induced trivialization $t_x:N^*Z\to\mathbb{R}$. If there is a contactomorphism $\Phi:Z\to Z$ such that \begin{enumerate}\item $\Phi^*(t_x)_*a=f\cdot(t_x)_*\widetilde{a}$ for positive $f\in\ C^\infty(Z)$,    
\item $\Phi^*b_1=\widetilde{b}_1$, and
\item $\Phi^*b_2|_Z=\widetilde{b}_2|_Z$, \end{enumerate} then there exists a neighborhood $U\supset Z$ and a scattering-symplectomorphism $\phi:U\to U$ such that $\phi^*\omega_1=\omega_2$.\end{proposition} 
\begin{proof} Fix a local $Z$ defining function $x$ and consider the induced trivialization $t_x:N^*Z\to\mathbb{R}$. Let $\alpha=\Phi^*(t_x)_*a$ and $\widetilde{\alpha}=(t_x)_*\widetilde{a}$. Let $\beta_1\in b_1$ and  $\beta_2\in b_2$. 
By Proposition \ref{sctube1}, there exists a tubular neighborhood $Z\times \left\{{x\in (-\delta,\delta)}\right\}$ on which

\centerline{$\displaystyle{\omega=\dfrac{dx}{x^3}\wedge (\alpha+x^2\beta_1)-\dfrac{d\alpha}{2x^2}+\beta_2.}$} 
\noindent Let $\Phi:Z\to Z$ be a contactomorphism such that $\Phi^*(\alpha)= e^g\widetilde{\alpha}$, for a smooth function $g\in\mathcal{C}^\infty(Z)$. Define a function 

\centerline{$f:Z\times (-\delta,\delta)\to\mathbb{R}~~\text{ by } ~~f(z,x)=\sqrt{e^{g(z)}}x.$} \noindent Then consider 

\centerline{$\Psi:Z\times(-\delta,\delta)\to Z\times(-\delta,\delta)~~ \text{ given by }~~\Psi(z,x)=(\Phi(z),f(z,x)).$} \noindent
Since 

\centerline{$\dfrac{d(e^{g/2}x)}{e^{3g/2}x^3}=\dfrac{dg}{2e^gx^2}+\dfrac{dx}{e^gx^3}$,} \noindent we have that 

\centerline{$\Psi^*(\omega)=\dfrac{dx}{x^3}\wedge\widetilde{\alpha}+\dfrac{dx}{x}\wedge\Phi^*\beta_1 -\dfrac{d\Phi^*\alpha}{2e^gx^2}+\Phi^*\beta_2 + \dfrac{dg}{2x^2}\wedge\widetilde{\alpha}+\dfrac{dg}{2}\wedge\Phi^*\beta_1.$} \noindent 
Because 

\centerline{$-\dfrac{d\Phi^*\alpha}{2e^gx^2}=-\dfrac{e^gdg}{2e^gx^2}\wedge\widetilde{\alpha}- \dfrac{e^gd\widetilde{\alpha}}{2e^gx^2},$} \noindent we have that 

\centerline{$\Psi^*(\omega)=\dfrac{dx}{x^3}\wedge\widetilde{\alpha}-\dfrac{d\widetilde{\alpha}}{2x^2}+\dfrac{dx}{x}\wedge\Phi^*\beta_1 +\Phi^*\beta_2 +\dfrac{dg}{2}\wedge\Phi^*\beta_1.$}
\noindent Note that $\Phi^*\beta_1$ is closed, so 

\centerline{$d\left(\dfrac{g}{2}\Phi^*\beta_1\right)=\dfrac{dg}{2}\wedge\Phi^*\beta_1$} \noindent is an exact form. Since $\Phi^*\widetilde{b_1}=b_1$ and $\Phi^*\widetilde{b}_2|_Z=b_2|_Z$, we have $\Phi^*\beta_1\in \widetilde{b_1}$ and  

\centerline{$\Phi^*\beta_2 +\dfrac{dg}{2}\wedge\Phi^*\beta_1\in\widetilde{b_2}$.} \noindent Thus by Proposition \ref{sctube1}, there exists a scattering-symplectomorphism between $\Psi^*\omega$ and $\widetilde{\omega}$ on a tubular neighborhood of $Z$.  \end{proof}

\subsubsection{Scattering-Symplectic Spheres} We will conclude this section by providing an example of scattering-symplectic manifolds. All even dimensional spheres admit scattering-symplectic structures. Let $(x_1,y_1,\dots,x_n,y_n,z)$ be global coordinates in $\mathbb{R}^{2n+1}$. Consider the sphere

\centerline{$\displaystyle{\mathbb{S}^{2n}=\left\{{\sum_{i=1}^{n}(x_i^2+y_i^2)+z^2=1}\right\}}$} \noindent with equator

\centerline{$\displaystyle{\mathbb{S}^{2n-1}=\left\{{\sum_{i=1}^n(x_i^2+y_i^2)=1, z=0}\right\}}.$} \vspace{1ex}

\begin{fig}  \begin{center} $\mathbb{S}^{2n}$ embedded in $\mathbb{R}^{n+1}$ \vspace{2ex}
 
\includegraphics[scale=.35]{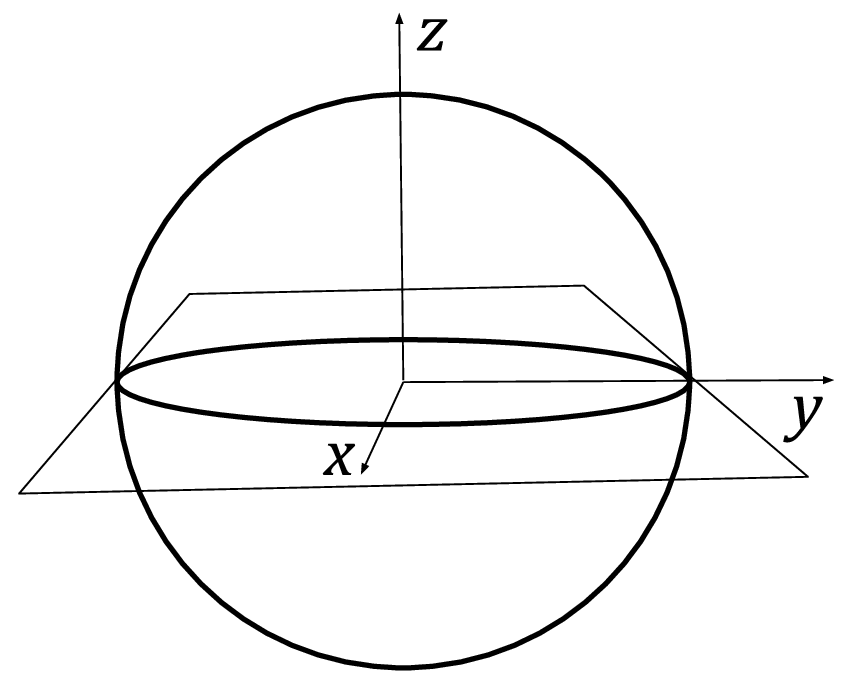}\end{center} \end{fig}

\noindent We define a one form $\displaystyle{\sigma=\frac{1}{2}\sum_{i=1}^{n}(x_idy_i-y_idx_i)}$. Consider the scattering form 

\centerline{$\displaystyle{\beta = -2\frac{dz}{z^3}\wedge\sigma+\frac{1}{z^2}d\sigma}$} \noindent restricted to $\mathbb{S}^{2n}$. 

\begin{proposition}\label{scspheres} $(\mathbb{S}^{2n},\mathbb{S}^{2n-1}, \beta)$ is a scattering-symplectic manifold. \end{proposition}

\begin{proof} First notice that $\displaystyle{\beta = d\left(\frac{\sigma}{z^2}\right)}$. Thus $\beta$ is closed. We point out that this does not make $\beta$ exact as a scattering form. We are left to show that $\beta$ is non-degenerate on the $2n$-sphere. In the set $U_{x_1}:=\mathbb{S}^{2n}\setminus\left\{{x_1=0}\right\}$, we have smooth coordinates $(y_1,x_2,y_2,\dots,x_n,y_n,z)$. By rewriting 

\centerline{$\displaystyle{x_1=\sqrt{1-y_1^2-\sum_{i=2}^n(x_i^2+y_i^2)-z^2}}$} \noindent and

\centerline{$\displaystyle{dx_1=\dfrac{-\Big(y_1dy_1+\displaystyle{\sum_{i=2}^n(x_idx_i+y_idy_i)}+zdz\Big)}{x_1}},$} \noindent we see that in $U_{x_1}$,

\centerline{$\displaystyle{\beta = -\dfrac{dz}{z^3}\wedge\Big(x_1dy_1+\dfrac{y_1^2dy_1}{x_1}+\dfrac{y_1}{x_1}\Big(\displaystyle{\sum_{i=2}^n(x_idx_i+y_idy_i)}\Big)+\sum_{i=2}^{n}(x_idy_i-y_idx_i)\Big)}$}

\centerline{$\displaystyle{+\dfrac{1}{z^2}\dfrac{-\Big(\displaystyle{\sum_{i=2}^n(x_idx_i+y_idy_i)}+zdz\Big)}{x_1}\wedge dy_1+\dfrac{1}{z^2}\displaystyle{\sum_{i=2}^{n}(dx_i\wedge dy_i)}.}$} The coefficient is 1 for terms of the form 

\centerline{$\dfrac{1}{z^2}dx_i\wedge dy_i$} \noindent for $i=2,\dots,n$. Thus to show non-degeneracy, it suffices to show that the coefficient of the term $\displaystyle{\frac{dz}{z^3}\wedge dy_1}$ is always nonzero. This coefficient is 

\centerline{$-\left(x_1+\dfrac{y_1^2}{x_1}+\dfrac{z^2}{x_1}\right)$.} \noindent Since, $x_1\neq 0$, this function is always nonzero in $U_{x_1}$. By symmetry, this argument shows that $\beta$ is non-degenerate in the sets  $U_{x_i}=\mathbb{S}^{2n}\setminus\left\{{x_i=0}\right\}$, and $U_{y_i}=\mathbb{S}^{2n}\setminus\left\{{y_i=0}\right\}$ for $i=1,\dots,n$. We are left to consider $\beta$ at the poles where $z$ is $\pm 1$, that is at the points $(0,0,0,\dots,0,0,\pm 1)$. Here, $\beta = \displaystyle{\sum_{i=1}^{n}dx_i\wedge dy_i}$. Thus $\beta$ is non-degenerate on the $2n$-sphere. 
\end{proof}

\begin{remark} For cohomological reasons, there are no symplectic spheres in dimensions greater than 2. Similarly, M\u{a}rcut and Orsono-Torres \cite{muarcuct2014cohomological} proved that a compact b-symplectic manifold $(M,Z)$ of dimension $2n$ has a class $c$ in $H^2(M)$ such that $c^{n-1}$ is nonzero in $H^{2n-2}(M)$. Thus there are also no b-symplectic spheres in dimensions greater than 2.\end{remark}

\begin{remark} This shows for a scattering symplectic structure to exist on a compact manifold $(M,Z)$, $Z$ must admit a co-orientable contact structure. Strikingly, it also shows that sometimes this is all you need!   
\end{remark} 

\subsection{Contact hypersurfaces} \label{Contact}
We have seen that every scattering-symplectic manifold has a co-oriented contact hypersurface. Now we will explore the opposite question: given a contact hypersurface, does it appear as the singular hypersurface of a scattering symplectic manifold? Given a co-oriented contact manifold $(Z,\alpha)$, there always exists a non-compact scattering symplectic manifold $(M,\omega)$ with $Z$ as singular hypersurface such that $\omega$ induces $\alpha$ on $Z$. 

\begin{proposition} Let $Z$ be a $2n-1$ dimensional contact manifold with globally defined contact form $\alpha$. Let $\widetilde{Z}=Z\times\mathbb{R}$ and let $\pi:\widetilde{Z}\to Z$ be the obvious projection. Let $x$ be a coordinate on $\mathbb{R}$. Then 

\centerline{$\displaystyle{\omega = d\left(\frac{\pi^*\alpha}{x^2}\right)}$} \noindent is a scattering symplectic form on $(\widetilde{Z},Z)$. We call $(\widetilde{Z},Z,\omega)$ the {\bf scattering symplectization} of $(Z,\alpha)$.  \end{proposition} 

\begin{proof}  It is clear by construction that $\omega$ is closed. This does not make $\omega$ exact as a scattering form. For ease of notation we will write $\alpha$ rather than $\pi^*\alpha$. Then 

\centerline{$\begin{displaystyle}\omega = -\frac{2}{x^3} dx\wedge \alpha+\frac{d_Z\alpha}{x^2} \end{displaystyle}.$} \noindent To check non-degeneracy of $\omega$, notice that 

\centerline{$\displaystyle{\omega^n=-\frac{2}{x^{2n+1}} dx\wedge \alpha\wedge(d_Z\alpha)^{n-1}}.$} \noindent A consequence of $Z$ being contact is that $\alpha\wedge(d_Z\alpha)^{n-1}\neq 0$. Thus $\omega$ is a closed non-degenerate scattering form on $(\widetilde{Z},Z)$.   
\end{proof}

In this construction, away from $\left\{{x=0}\right\}$, the one form 

\centerline{$\dfrac{\alpha}{x^2}$} \noindent is a smooth primitive for $\omega$. In particular, the vector field 

\centerline{$V=-\dfrac{x}{2}\dfrac{\partial }{\partial x}$} \noindent is non-zero when $x\neq 0$, is transverse to each level set $Z = \left\{{x=c}\right\}$ for nonzero constants $c\in\mathbb{R}$, and satisfies 

\centerline{$i_V\omega=\alpha.$} \noindent  
This precisely means that $\widetilde{Z}\setminus\left\{{|x|<\varepsilon}\right\}$, for any $\varepsilon>0$, is a disjoint union of strong symplectic fillings of the contact manifold $(Z,\alpha)$. 

This additional structure - a Liouville vector field $V$ giving this relation between $\omega$ and $\alpha$ - is not a feature of all scattering-symplectic manifolds. Indeed, we will now show that whether or not a scattering-symplectic manifold is a strong symplectic filling in this sense can be read off of the sc- cohomology class of the scattering symplectic form. 

Recall from Section \ref{scatteringsymplecticgeometry}, that we can associate to $[\omega]\in {}^{sc}H^2(M)$ a decomposition $(a,b_1,b_2)$  where $a\in \Omega^{1}(Z;(N^*Z)^{-2})$, $b_1\in H^1(Z)$, and $b_2\in H^2(M)$. There is a tubular neighborhood of $Z$ with a given local $Z$ defining function $x$ giving us a trivialization $t_x:N^*Z\to \mathbb{R}$. By Proposition \ref{sctube1}, this defines a smooth contact form $\alpha=(t_x)_*(a)\in\Omega^1(Z)$, $\beta_1\in b_1$, and $\beta_2\in b_2$ such that 

\centerline{$\displaystyle{\omega=\dfrac{dx}{x^3}\wedge (\alpha+x^2\beta_1)-\dfrac{d\alpha}{2x^2}+\beta_2.}$} \noindent In the following propositions we will always be working with $\omega$ in such a tubular neighborhood. 

\begin{proposition}\label{notfill} Let $(M,Z,\omega)$ be a scattering-symplectic manifold with singular contact hypersurface $(Z,\alpha)$. If $[\omega]$ has cohomology decomposition $(a,b_1,b_2)$ with $b_1$ or $b_2 \neq 0$, then for $\varepsilon>0$ small, $(M\setminus \left\{|x|<\varepsilon\right\},\omega)$ is not a strong symplectic filling of $(Z,\alpha)$.   
\end{proposition}

The following lemma gives us a normal form for strong convex and strong concave fillings in a neighborhood of the boundary. 

\begin{lemma} \label{gluelemma} If $(M,\omega)$ is a strong convex symplectic filling of $(Z,\xi)$, then for some $c>0$, there exists a collar neighborhood $Z\times [0,c)_r$ of $Z$ on which 

\centerline{$\omega=d(e^{-r}\alpha)$} \noindent and, given the projection $p:Z\times [0,c)\to Z$, $\alpha=p^*(\widetilde{\alpha})$ for an $\widetilde{\alpha}$ satisfying $\ker\widetilde{\alpha}=\xi$. 

If $(M,\omega)$ is a strong concave symplectic filling of $(Z,\xi)$, then for some $c>0$, there exists a collar neighborhood $Z\times [0,c)_r$ of $Z$ on which 

\centerline{$\omega=d(e^{r}\alpha)$} \noindent and, given the projection $p:Z\times [0,c)\to Z$, $\alpha=p^*(\widetilde{\alpha})$ for an $\widetilde{\alpha}$ satisfying $\ker\widetilde{\alpha}=\xi$. 
\end{lemma} 

\begin{proof} By definition of convex strong symplectic filling, near $Z$ there exists a nowhere vanishing vector field $v$ transverse to $Z$ such that $\mathcal{L}_{-v}\omega=\omega$. By using the flow of $\vec{v}$, we can choose a collar neighborhood $Z\times  [0,c)_r$ of $Z$ such that $r$ is the coordinate for $[0,c)$, $c>0$, and  

\centerline{$\displaystyle{v|_{Z\times [0,c)}=\frac{\partial}{\partial r}.}$}

It follows from the definition of strong symplectic filling that 

\centerline{$\displaystyle{(i_{-\frac{\partial}{\partial r}}\omega)|_Z= \alpha}$} \noindent for some contact form $\alpha$ defining $\xi$ on $Z$. Given the projection 

\centerline{$p:Z\times [0,c)\to Z,$} \noindent let $\widetilde{\alpha}=p^*\alpha$. Let 

\centerline{$\gamma = i_{-\frac{\partial}{\partial_r}}\omega.$} \noindent Then because $\mathcal{L}_{-v} \omega = \omega$, we have $d\gamma = \omega$. 

Let $x_1,x_2,\dots,x_n$ be any set of coordinates in $Z$. Then in these coordinates, 

\centerline{$\displaystyle{\gamma = gdr+\sum_{i}f_idx_i}$} \noindent for some smooth functions $g,f_i\in\mathcal{C}^{\infty}(Z\times [0,c))$. We can compute 

\centerline{$\displaystyle{d\gamma = \sum_j \left( \frac{\partial g}{\partial x_j}dx_j\wedge dr +\sum_{i}  \frac{\partial f_i}{\partial x_j}dx_j\wedge dx_i\right)+\sum_i \frac{\partial f_i}{\partial r}dr\wedge dx_i}.$} 

Then $\displaystyle{i_{-\frac{\partial}{\partial_r}}d\gamma = \gamma}$ gives the relations 

\centerline{$\displaystyle{ \sum_i  \left(\frac{\partial g}{\partial x_i} - \frac{\partial f_i}{\partial r}\right) dx_i=gdr+\sum_{i}f_idx_i}.$} \noindent In other words, 

\centerline{$g=0,\text{ and}-\frac{\partial f_i}{\partial r}=f_i.$} \noindent Thus $f_i=c_i\cdot e^{-r}$ for some functions $c_i$ constant with respect to $r$. Since $\gamma|_{\left\{r=0\right\}}=\alpha$, we have that 

\centerline{$\displaystyle{\gamma = e^{-r}\widetilde{\alpha}\text{ and }\omega=d(e^{-r}\widetilde{\alpha}).}$} 

A similar computation shows the statement for a concave filling. 
\end{proof} 

 We are now prepared to prove Proposition \ref{notfill}. 

\begin{proof}  Consider a tubular neighborhood $\tau=Z\times(-\varepsilon,\varepsilon)_x$ of $Z$. Assume, for a contradiction, that $M\setminus \left\{{x<\varepsilon}\right\}$ for $\varepsilon>0$ is a strong symplectic filling of $(Z,\alpha)$. By Lemma \ref{gluelemma}, there is a smooth function $f\in\mathcal{C}^\infty(\tau)$ such that  

\centerline{$\displaystyle{\omega=d(f\alpha)=\partial_xfdx\wedge\alpha+d_Zf\wedge\alpha+fd\alpha.}$} \noindent 

By the proof of Proposition \ref{sctube1}, there exists a tubular neighborhood $Z\times(-\widetilde{\varepsilon},\widetilde{\varepsilon})_x$ with $\widetilde{\varepsilon}$ chosen smaller than $\varepsilon$ as necessary such that    

\centerline{$\displaystyle{\omega=\dfrac{dx}{x^3}\wedge (\alpha+x^2\beta_1)-\dfrac{d\alpha}{2x^2}+\beta_2}$} \noindent for $\alpha,\beta_1,\beta_2\in\Omega^*(Z)$.

Thus 

\centerline{$\displaystyle{\dfrac{dx}{x^3}\wedge(\alpha+\beta_1x^2)=\partial_xfdx\wedge\alpha\text{ and }\dfrac{\alpha}{x^3}+\dfrac{\beta_1}{x}=\partial_xf\alpha.}$} \noindent  We can solve for $\beta_1$, 

\centerline{$\displaystyle{\beta_1=(x\partial_xf-\dfrac{1}{x^2})\alpha.}$} \noindent Since $\beta_1$ is closed, 

\centerline{$\displaystyle{0=d\beta_1=(x\partial_xf-\dfrac{1}{x^2})d\alpha+(\partial_xf+x\partial_{xx}f+\dfrac{2}{x^3})dx\wedge\alpha+xd_Z(\partial_xf)\wedge\alpha.}$} \noindent Thus 

\centerline{$\displaystyle{\partial_xf+x\partial_{xx}f=-\dfrac{2}{x^3}}$} \noindent with solution $f=-\dfrac{1}{2x^2}$. 

Then 

\centerline{$\displaystyle{\omega=d(-\dfrac{1}{2x^2}\alpha)=\dfrac{dx}{x^3}\wedge\alpha-\dfrac{d\alpha}{2x^2}\text{ and }\beta_1=\beta_2=0.}$} \noindent Thus we have reached a contradiction. \end{proof}

On the other hand, if $b_1,b_2=0$ in a cohomological decomposition of a scattering symplectic form $[\omega]$, then we always have the additional structure of a strong symplectic filling. And in fact, that filling is always convex, meaning the Liouville vector field points outward at $Z$.  

\begin{proposition} Let $(M,Z,\omega)$ be a scattering-symplectic manifold with singular contact hypersurface $(Z,\alpha)$ separating $M$. If $[\omega]$ has cohomology decomposition $(a,0,0)$, then for $\varepsilon>0$ small, $(M\setminus \left\{|x|<\varepsilon\right\},\omega)=(M_{x\geq \varepsilon},\omega)\cup(M_{x\leq \varepsilon},\omega)$ is a collection of symplectic manifolds  each with contact boundary $(Z,\alpha)$ such that $\omega$ is a convex strong symplectic filling of $\alpha$.\end{proposition} 

\begin{proof} Choose a tubular neighborhood $\tau=Z\times(-\varepsilon,\varepsilon)_x$ of $Z$ as in Proposition \ref{sctube1}. Then 

\centerline{$\displaystyle{\omega|_\tau=\dfrac{dx}{x^3}\wedge\alpha-\dfrac{d\alpha}{2x^2}.}$} \noindent Define $M_{x\geq \varepsilon}$ to be the connected component of $M\setminus (Z\times (-\varepsilon,\varepsilon))$ containing $Z\times \left\{{\varepsilon}\right\}$ and let its symplectic form be the scattering symplectic form on $M$ restricted to $M_{x\geq \varepsilon}$. Similarly, define $M_{x\leq\varepsilon}$ to be the connected component of $M\setminus (Z\times (-\varepsilon,\varepsilon))$ containing $Z\times \left\{{-\varepsilon}\right\}$ and let its symplectic form be the scattering symplectic form on $M$ restricted to $M_{x\leq \varepsilon}$. 

Let $V=-\dfrac{x}{2}\partial_x$. Notice for all points in $Z\times \left\{{\varepsilon}\right\}$ and $Z\times \left\{{-\varepsilon}\right\}$ that $V$ is tranverse to $Z$. Next, observe that $i_V\omega =\dfrac{\alpha}{x^2}$. Thus $\mathcal{L}_V\omega=di_V\omega=\omega$. Notice that $\omega$ on $M_{x\geq \varepsilon}$ and $\omega$ on $M_{x\leq \varepsilon}$ induce the same contact form. In particular, 

\centerline{$\displaystyle{-\dfrac{\alpha}{2x^2}\Big|_{Z\times \left\{{\varepsilon}\right\}}=-\dfrac{\alpha}{2(\varepsilon)^2}=-\dfrac{\alpha}{2(-\varepsilon)^2}=-\dfrac{\alpha}{2x^2}\Big|_{Z\times \left\{{-\varepsilon}\right\}}.}$}

Further, notice that $V\big|_{Z\times \left\{{\varepsilon}\right\}}=-\dfrac{\varepsilon}{2}\partial_x$ is an outward pointing vector. Similarly, 

\centerline{$\displaystyle{V\big|_{Z\times \left\{{-\varepsilon}\right\}}=-\dfrac{-\varepsilon}{2}\partial_x}$} \noindent is outward pointing. Thus both fillings are convex. 
\end{proof} 

\subsection{Symplectic Gluing} \label{gluingsection}
We will demonstrate how to glue strong symplectic fillings along a common boundary. In particular, by allowing scattering-symplectic structures, we can glue convex fillings to convex fillings and by allowing folded-symplectic structures, we can glue concave fillings to concave fillings. 

 We will begin by recalling how a strong concave and strong convex symplectic filling are glued to form a symplectic manifold; see for example Theorem 5.4, \cite{mrowka2009low}. 

\begin{proposition}
Let $(M_1,\omega_1)$ and $(M_2,\omega_2)$ be a strong convex and strong concave symplectic filling respectively of $(Z,\xi)$. Then $\displaystyle{M_1\cup_Z M_2}$, the union of $M_1$ to $M_2$ at $Z$, has a symplectic structure $\omega$ such that $\omega|_{N_1}=\omega_1|_{N_1}$ and $\omega|_{N_2}=\omega_2|_{N_2}$ where $N_i=M_i\setminus U_i$ for a tubular neighborhood $U_i$ of $Z$ in $M_i$.\end{proposition}

\begin{proof} By Lemma \ref{gluelemma}, since $\omega_1$ is a strong convex symplectic filling of $(Z,\xi)$, there exists  $Z\times [0,c_1)$ a tubular neighborhood of $Z$ in $M_1$ on which $\omega_1= d(e^{-r_1}\alpha)$ for $r_1$ the coordinate for the interval $[0,c_1)$ and where $\ker\alpha=\xi$. Similarly, because $(M_2,\omega_2)$ is a strong concave symplectic filling, we have $Z\times [0,c_2)$ a tubular neighborhood of $Z$ in $M_2$ on which $\omega_2= d(e^{r_2}\alpha)$ for $r_2$ the coordinate of the interval $[0,c_2)$. Without loss of generality, assume $c_1=c_2=c$. 

We attach a collar neighborhood $Z\times (-c,0)_{r_1}$ to $M_1$ and a collar neighborhood $Z\times (-c,0)_{r_2}$ to $M_2$. The union $M_1\cup_Z M_2$ is formed from identifying  $Z\times (-c,c)_{r_1}$ with $Z\times (-c,c)_{r_2}$ by mapping $Z$ to itself and setting $r_1=-r_2=r$. The smooth structure on $M_1\cup_Z M_2$ is obtained from the charts on $M_1$ and $M_2$ respectively. 

We define a symplectic form on $M_1\cup_Z M_2$ by $\omega=d(e^r\alpha)$. We interpret $\omega$ to extend as $\omega_2$ into $M_2\setminus (Z\times [0,c)_{r_2})$.  Similarly, we interpret $\omega$ to extend as $\omega_1$ into $M_1\setminus (Z\times [0,c)_{r_1})$. \end{proof} 

Next, we introduce a method for gluing two strong convex symplectic fillings by using a scattering-symplectic structure.  

\begin{theorem}\label{gluingtheorem}Let $(M_1,\omega_1)$ and $(M_2,\omega_2)$ be strong convex symplectic fillings of $(Z,\xi)$. Then $\displaystyle{M_1\cup_Z M_2}$, the union of $M_1$ to $M_2$ at $Z$, has a scattering symplectic structure $\omega$ such that $\omega|_{N_1}=\omega_1|_{N_1}$ and $\omega|_{N_2}=\omega_2|_{N_2}$ where $N_i=M_i\setminus U_i$ for a tubular neighborhood $U_i$ of $Z$ in $M_i$, and the singular hypersurface of $\omega$ is $Z$. 
\end{theorem} 

\begin{proof} Begin by gluing $M_1$ and $M_2$ along $Z$. By Lemma \ref{gluelemma}, since $\omega_1$ and $\omega_2$ are both strong convex symplectic fillings of $(Z,\xi)$, there exists  a tubular neighborhood $U \cong Z\times (-\varepsilon,\varepsilon)_r$ of $Z$ in $M_1\cup_Z M_2$ where $\displaystyle{\omega_1\big|_{(Z\times (-\varepsilon,0)} = d(e^{|-r|}\alpha)}$ and where  $\omega_2\big|_{(Z\times (0,\varepsilon)} = d(e^{|-r|}\alpha)$. We define a scattering symplectic 2-form $\omega$ on $M_1\cup_Z M_2$ by $\omega = d(f(|r|)\alpha)$, where $f:(0,\infty)\to (0,\infty)$ is a smooth map satisfying $f(t) = e^{-t}$ for $t\geq \varepsilon/2$, $f(t) = 1/t^2$ for $t\geq \varepsilon/4$ and $f^\prime\leq 0$. By definition, the scattering form $\omega$ is closed. Since $f^\prime(t)\leq 0$, we also have that $\omega$ is non-degenerate. \end{proof}

\begin{remark} Because $\mathbb{D}^{2n}$ with the standard symplectic form $\omega_{st}$ is a strong convex symplectic filling of the unit sphere $\mathbb{S}^{2n-1}$ with standard contact structure $\xi_{st}$, Theorem \ref{gluingtheorem} provides an alternate way of constructing the scattering symplectic spheres described in Proposition \ref{scspheres}.

\begin{fig} $\mathbb{D}^{2n}$ glues with $\mathbb{D}^{2n}$ to yield scattering symplectic $(\mathbb{S}^{2n},\mathbb{S}^{2n-1})$. 

\begin{center}\includegraphics[scale=.27]{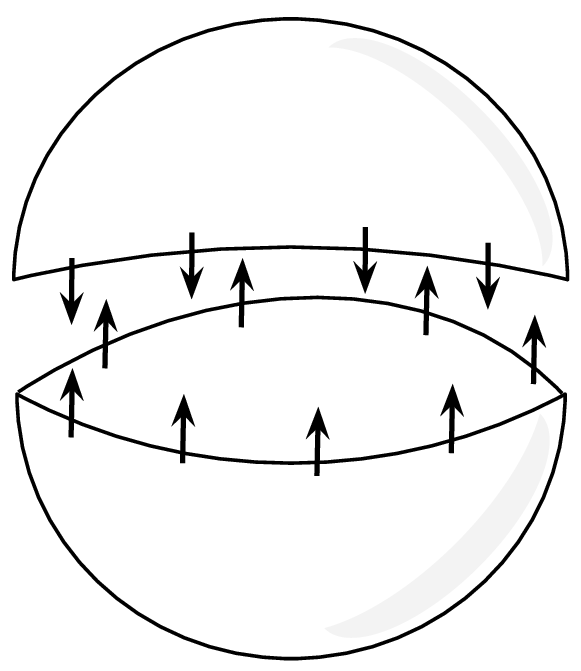}\hspace{10ex}  \includegraphics[scale=.27]{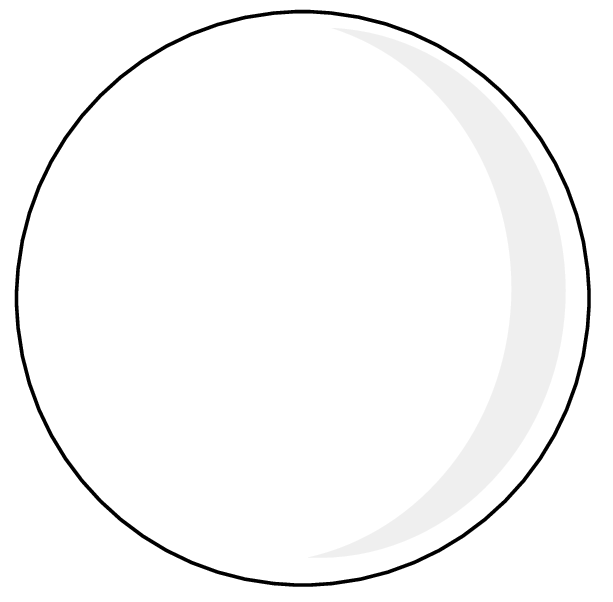}\end{center} \end{fig}

\end{remark} 

Theorem \ref{gluingtheorem} provides us with a treasure trove of additional examples, particularly in dimension 4 where constructing strong convex symplectic fillings has been an industry in its own right. For a certainly incomplete list, see \cite{eliashberg1996unique,geiges1995examples,ghiggini2005strongly}, or \cite{mcduff1991symplectic}. For the sake of brevity, we will limit our attention to a couple of examples. 

\begin{example} {\bf The pair $(\mathbb{T}^{2}\times\mathbb{S}^2,\mathbb{T}^{3})$ is scattering symplectic}.

Let $(q_1,q_2)$ be coordinates for the torus $\mathbb{T}^2$ and let $(p_1,p_2)$ be coordinates for the disk $\mathbb{D}^2$. Then $\omega=dp_1\wedge dq_1 +dp_2\wedge dq_2$ is a symplectic form on $\mathbb{T}^2\times\mathbb{D}^2$. We can rewrite this expression using polar coordinates by setting $p_1=r\cos\theta$ and $p_2=r\sin\theta$. Then $\gamma = r\cos\theta~dq_1+r\sin\theta~dq_2$ is a primitive for $\omega$ near the boundary $\partial(\mathbb{T}^2\times\mathbb{D}^2)$. At the boundary, $\gamma|_{r=1}=\cos\theta~dq_1+\sin\theta~dq_2$ is a contact form on $\mathbb{T}^2\times\mathbb{S}^1=\mathbb{T}^3$. Notice that $i_{r\partial_r}d\gamma=\gamma$ for outward pointing normal vector $r\partial_r$. 

\begin{fig}  $\mathbb{T}^2\times\mathbb{D}^2$ glues with $\mathbb{T}^2\times\mathbb{D}^2$ to yield scattering symplectic $(\mathbb{T}^{2}\times\mathbb{S}^2,\mathbb{T}^{3})$.

\begin{center}
\includegraphics[scale=.25]{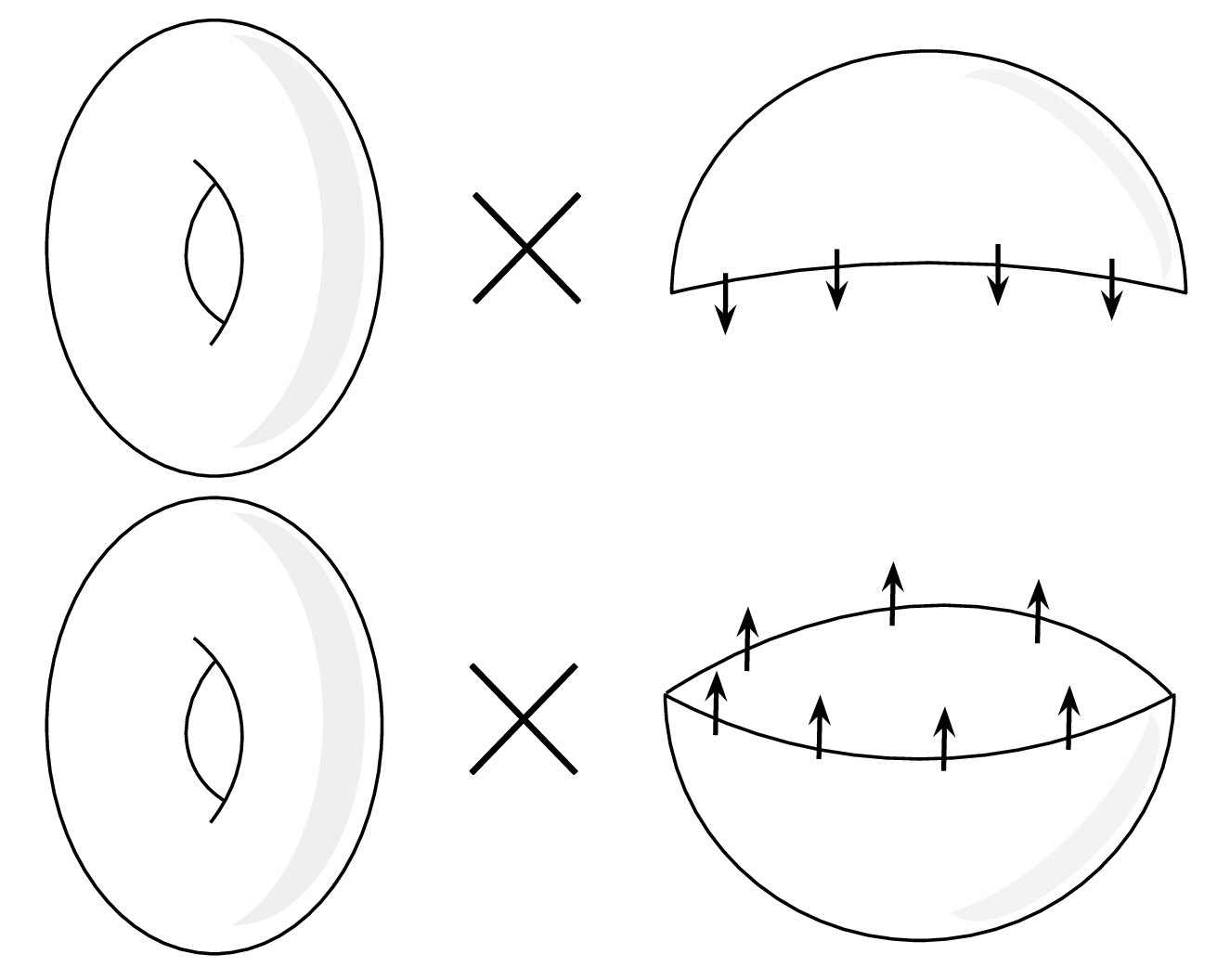}\hspace{10ex}\includegraphics[scale=.24]{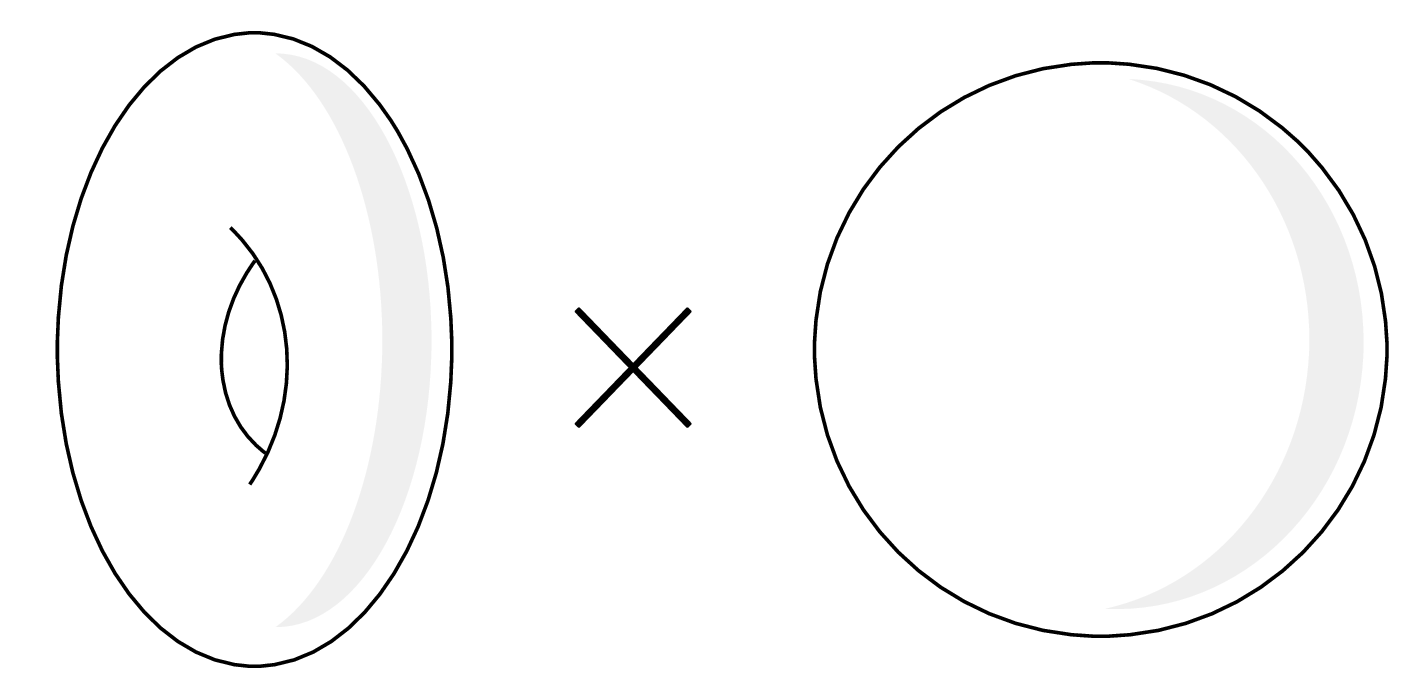}\end{center}
\end{fig}

Thus $(\mathbb{T}^2\times\mathbb{D}^2, \omega)$ is a strong convex symplectic filling of the torus $\mathbb{T}^3$ with contact structure $\gamma|_{r=1}=0$. As described in Theorem  \ref{gluingtheorem}, we construct the union of $\mathbb{T}^2\times\mathbb{D}^2$ with itself at $\partial(\mathbb{T}^2\times\mathbb{D}^2)$. Thus the pair $(\mathbb{T}^2\times\mathbb{S}^2,\mathbb{T}^3)$, where $\mathbb{T}^3$ is identified as $\mathbb{T}^2\times\mathbb{S}^1$ and the factor $\mathbb{S}^1$ is the equator of $\mathbb{S}^2$, admits a scattering symplectic structure.  

\end{example}

\begin{example} {\bf The pair $(\mathbb{S}^3\times\mathbb{S}^1,\mathbb{S}^2\times\mathbb{S}^1)$ is scattering symplectic}.  

Let $(x, y, z)$ be the standard Euclidean coordinates for $\mathbb{D}^3$ and let $\theta$ be a coordinate for $\mathbb{S}^1$. The manifold $\mathbb{D}^3\times\mathbb{S}^1$ admits the symplectic form $\omega=2dx\wedge dy +dz\wedge d\theta$. Then $\gamma = xdy-ydx+zd\theta$ is a primitive for $\omega$. 
The vector field  $R=\dfrac{x}{2}\partial_x+\dfrac{y}{2}\partial_y+z\partial_z$ is an outward pointing normal vector at the boundary $\partial(\mathbb{D}^3\times\mathbb{S}^1)$. 
Further $i_Rd\gamma=\gamma$ and $\gamma|_{\partial(\mathbb{D}^3\times\mathbb{S}^1)}$ defines a contact structure on $\mathbb{S}^2\times\mathbb{S}^1$.  Thus $(\mathbb{D}^3\times\mathbb{S}^1, \omega)$ is a strong convex symplectic filling of $\mathbb{S}^2\times\mathbb{S}^1$ with contact structure $\gamma|_{\partial(\mathbb{D}^3\times\mathbb{S}^1)}$. 

\begin{fig}  $\mathbb{D}^3\times\mathbb{S}^1$ glues with $\mathbb{D}^3\times\mathbb{S}^1$ to yield scattering symplectic $(\mathbb{S}^3\times\mathbb{S}^1,\mathbb{S}^2\times\mathbb{S}^1)$.

\begin{center}

\includegraphics[scale=.2]{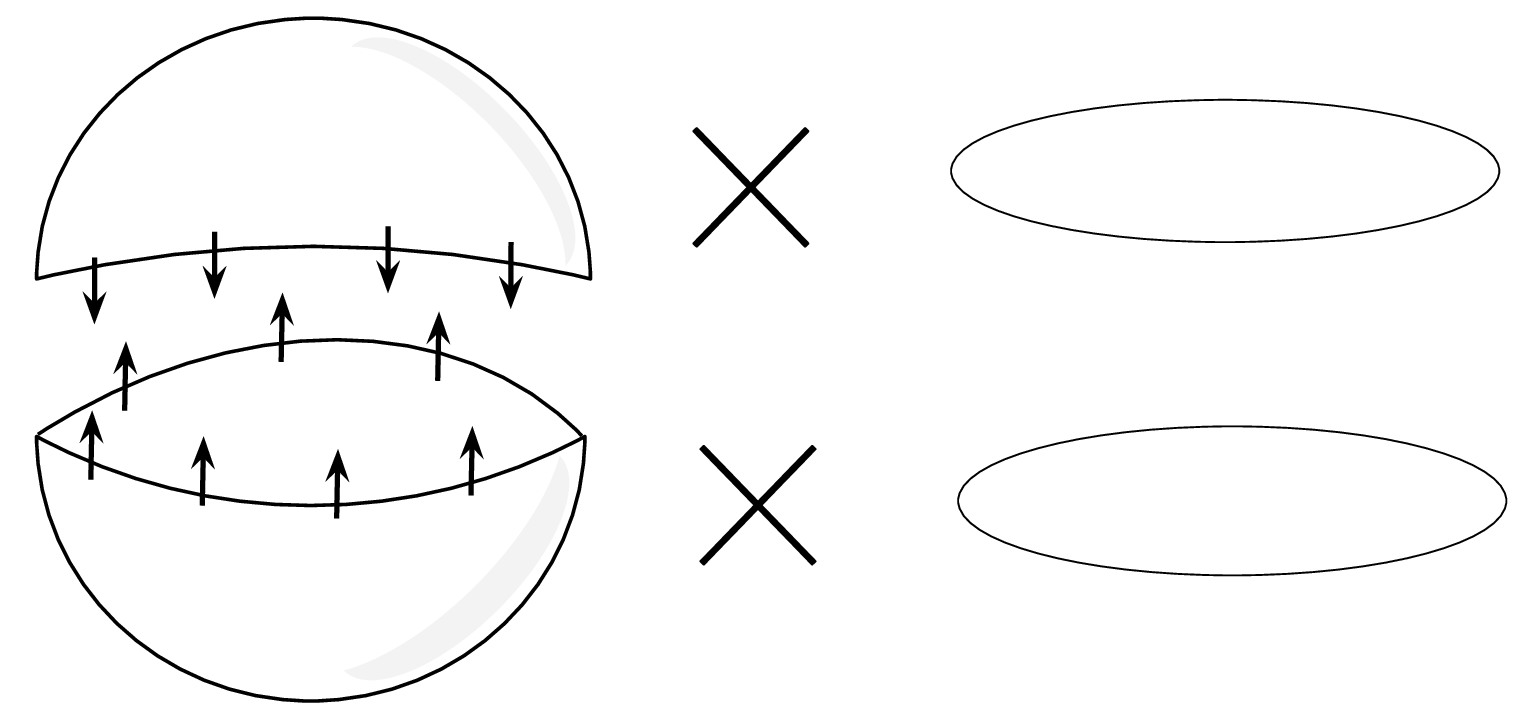}\hspace{10ex}\includegraphics[scale=.22]{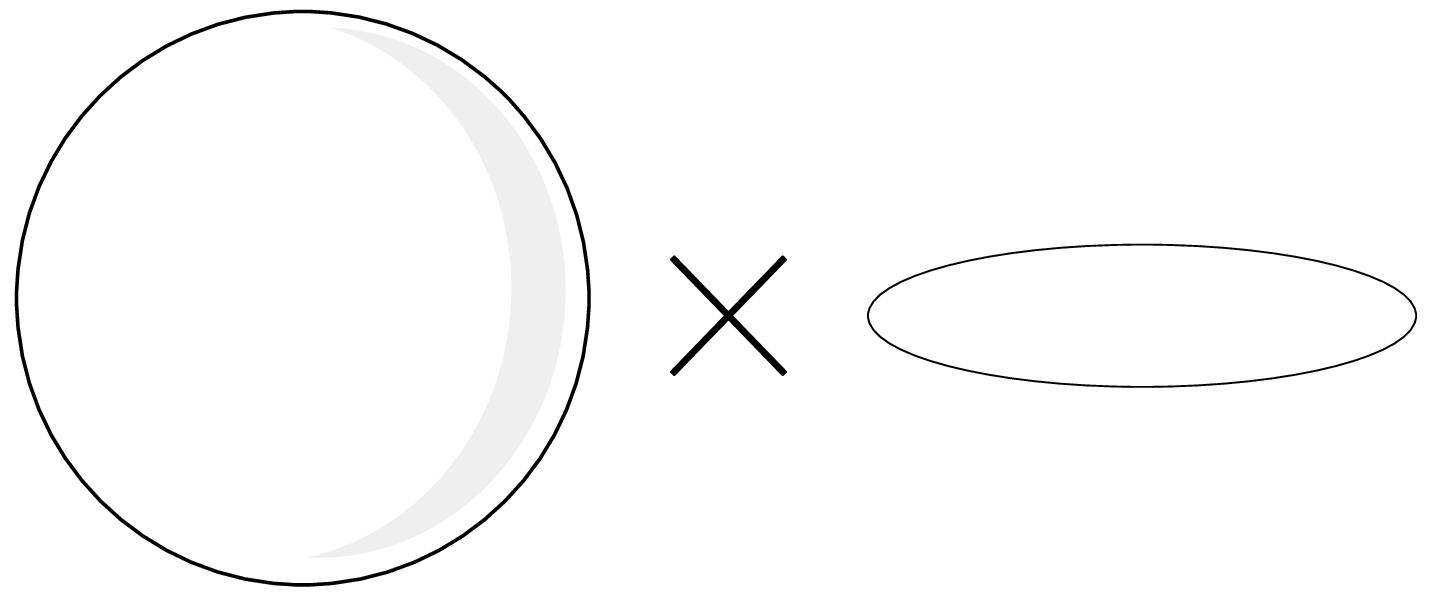}\end{center}\end{fig}

By Theorem  \ref{gluingtheorem}, we construct the union of $\mathbb{D}^3\times\mathbb{S}^1$ with itself at $\partial(\mathbb{D}^3\times\mathbb{S}^1)$. Then the pair $(\mathbb{S}^3\times\mathbb{S}^1,\mathbb{S}^2\times\mathbb{S}^1)$, where the factor $\mathbb{S}^2$ in $\mathbb{S}^2\times\mathbb{S}^1$ is identified as the equator of $\mathbb{S}^3$, admits a scattering symplectic structure.

\end{example} 
 
By expanding the symplectic category to allow scattering symplectic structures, we can overcome the obstacle preventing convex fillings from being glued to other convex fillings. In the next theorem, we show that folded symplectic structures can similarly overcome this obstacle for concave fillings, allowing a concave filling to be glued to another concave filling. 

Recall that a folded symplectic manifold $(M^{2n},Z,\omega)$ is a $2n$-dimensional manifold $M$ equipped with a closed two-form $\omega$ that is non-degenerate except on a hypersurface $Z$, called the folding hypersurface, where there exist
coordinates such that locally $Z=\left\{x_1=0\right\}$ and 
\begin{center}$\displaystyle{\omega=x_1dx_1\wedge dy_1+\sum_{i=2}^n dx_i\wedge dy_i.}$\end{center}
In section 6 of \cite{da2000unfolding}, Ana Cannas da Silva, Victor Guillemin, and Christopher Woodward prove that given any two compact oriented 2d-dimensional symplectic
manifolds $W_1$, $W_2$ with common boundary $Z$, their union over their boundary $W_1\cup_Z W_2$ admits a folded-symplectic structure. We will consider the special case when the two manifolds $W_1$ and $W_2$ are strong concave symplectic fillings of a contact boundary $(Z,\alpha)$. We will prove that $W_1\cup_Z W_2$ can be endowed with a folded-symplectic structure that preserves this strong concavity on either side of the hypersurface $Z$. 
\begin{theorem}\label{foldedgluingtheorem} Let $(M_1,\omega_1)$ and $(M_2,\omega_2)$ be strong concave symplectic fillings of $(Z,\xi)$. Then $\displaystyle{M_1\cup_Z M_2}$, the union of $M_1$ to $M_2$ at $Z$, has a folded symplectic structure $\omega$ such that $\omega|_{N_1}=\omega_1|_{N_1}$ and $\omega|_{N_2}=\omega_2|_{N_2}$ where $N_i=M_i\setminus U_i$ for a tubular neighborhood $U_i$ of $Z$ in $M_i$, and the folding hypersurface of $\omega$ is $Z$.
\end{theorem} 

This construction follows from an argument of Cannas da Silva,  Guillemin, and Woodward; see Example (2), Section 3 in \cite{da2000unfolding} with the added assumptions that $\omega_1|_Z=\omega_2|_Z$ and that the orientation induced on $\ker(\omega_i|_Z)$ is the same.  

Because $\mathbb{R}^{2n}\setminus\mathbb{D}^{2n}$ with the standard symplectic form is a concave symplectic filling of $\mathbb{S}^{2n-1}$ with the standard contact form, this construction immediately gives us a folded-symplectic connect sum over the sphere with its standard contact structure. 

By combining Theorems \ref{gluingtheorem} and \ref{foldedgluingtheorem}, we can construct examples of scattering-folded symplectic manifolds: 

\begin{definition} A \emph{scattering-folded symplectic manifold} $(M,Z_{sc},Z_f,\omega)$ is a manifold $M$, with two distinct hypersurfaces $Z_{sc}$ and $Z_f$, equipped with a two form $\omega$ that is symplectic everywhere except for on a singular hypersurface $Z_{sc}$ where it is a scattering symplectic form and on a folding hypersurface $Z_f$ where it is a folded symplectic form. 
\end{definition} 

\begin{example} Recall \cite{audin2012symplectic} that a \emph{symplectic cone} is a triple $(M,\omega,X)$ of a manifold $M$ with a symplectic form $\omega$ and a vector field $X$ such that $X$ generates a proper action of the real numbers on $M$ and $\mathcal{L}_X\omega=\omega$. Any symplectic cone is of the form $B\times\mathbb{R}$ where $B$ is a co-oriented contact manifold with contact form $\alpha$. In fact, all symplectic cones can be written as $(B\times\mathbb{R},d(e^t\alpha),\frac{\partial}{\partial t})$ where $t$ is the $\mathbb{R}$ coordinate. 

Let $(B\times\mathbb{R},d(e^t\alpha),\frac{\partial}{\partial t})$ be any symplectic cone of a co-oriented contact manifold $(B,\alpha)$. We will truncate $B\times\mathbb{R}$ by considering $B\times[-k,k]$ for any real number $k>0$. 

\begin{fig} Truncated symplectic cone
\begin{center}\includegraphics[scale=.7]{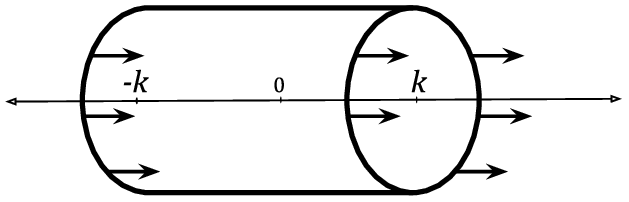}\end{center}\end{fig}

At the level set $t=k$, the vector field $\frac{\partial}{\partial t}$ is outward pointing and satisfies $i_\frac{\partial}{\partial t}\omega=\alpha$. Thus at $t=k$, we have a strong convex symplectic filling of $(B,\alpha)$. At the level set $t=-k$, the vector field $\frac{\partial}{\partial t}$ is inward pointing and satisfies $i_\frac{\partial}{\partial t}\omega=\alpha$. Thus at $t=-k$, we have a strong concave symplectic filling of $(B,\alpha)$. 

By Theorems \ref{gluingtheorem} and \ref{foldedgluingtheorem}, we can glue $B\times[-k,k]$ to a copy of itself to form the scattering-folded symplectic manifold $(B\times\mathbb{S}^1, B\times \left\{{k}\right\}, B\times\left\{{-k}\right\},\omega)$. 
\end{example} 

Our next example provides an explicit description of a scattering-folded symplectic manifold.  

\begin{example} {\bf The triple $\displaystyle{(\mathbb{T}^{2n},\cup_{2m}\mathbb{T}^{2n-1}, \cup_{2m}\mathbb{T}^{2n-1})}$ is scattering-folded symplectic for all $n,m\in\mathbb{N}$.}

In \cite{bourgeois2002odd}, Fr\'{e}d\'{e}ric Bourgeois defines a contact form $\beta$ on all odd dimensional tori $\mathbb{T}^{2n-1}$. 
We identify $\mathbb{T}^{2n}$ with $\mathbb{T}^{2n-1}\times \mathbb{S}^1$ and denote the angular coordinate on $\mathbb{S}^1$ by $\theta$. We define a scattering-folded form on $\mathbb{T}^{2n}$ by 

\centerline{$\displaystyle{\omega_m=d\left(\dfrac{\beta}{\sin^2(m\theta)}\right)=\dfrac{-2m\cos(m\theta)~d\theta\wedge\beta}{\sin^3(m\theta)} + \dfrac{d\beta}{\sin^2(m\theta)}}$} \noindent for any $m\in\mathbb{N}$. 

Then for each zero $z$ of $\sin(m\theta)$ in $[0,2\pi)$, we have a singular hypersurface $\mathbb{T}^{2n-1}\times \left\{{z}\right\}$. Since there are $2m$ zeroes of $\sin(m\theta)$ in $[0,2\pi)$, $Z_{sc}=\cup_{2m}\mathbb{T}^{2n-1}$. Similarly, for each zero $z$ of $\cos(m\theta)$ in $[0,2\pi)$, we have a folding hypersurface $\mathbb{T}^{2n-1}\times \left\{{z}\right\}$. Since there are $2m$ zeroes of $\cos(m\theta)$ in $[0,2\pi)$, $Z_{f}=\cup_{2m}\mathbb{T}^{2n-1}$. 

\begin{fig} $\displaystyle{(\mathbb{T}^{2},\cup_{4}\mathbb{S}^1}, \cup_{4}\mathbb{S}^{1})$. 

\begin{center}\hspace{5ex}\begin{minipage}{.3\linewidth}

\begin{center}\includegraphics[scale=.4]{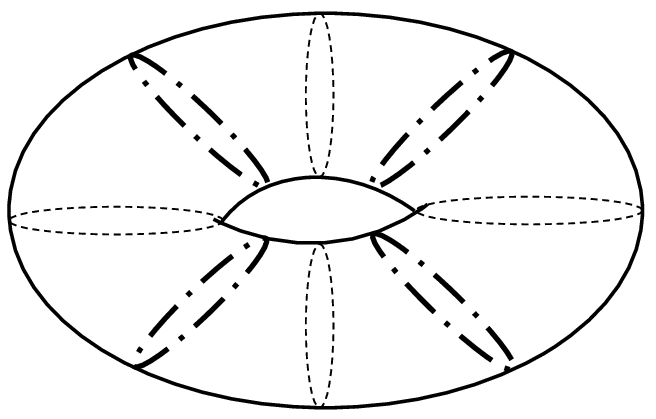}\end{center}\end{minipage}\hspace{5ex}\begin{minipage}{.15\linewidth} 

\begin{center}
folding 

singular \end{center}\end{minipage}\begin{minipage}{.3\linewidth}

\includegraphics[scale=.5]{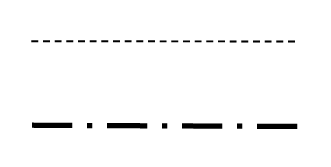}

\end{minipage}\end{center}

We equip $\mathbb{T}^{2}$ with the form 

\centerline{$d\left(\dfrac{d\theta_1}{\sin^2(2\theta_2)}\right)=\dfrac{-4\cos(2\theta_2)~d\theta_2\wedge d\theta_1}{\sin^3(2\theta_1)}.$}

The folding hypersurfaces occur at $\theta_2=\pi/4,3\pi/4, 5\pi/4, $ and $7\pi/4$. The singular hypersurfaces occur at $\theta_2=0,\pi/2, \pi,$ and $3\pi/2$. 
\end{fig} 

\end{example}

\subsection{Scattering Poisson geometry} 

A scattering-Poisson structure is dual to a scattering-symplectic structure. In this section we will explore these structures utilizing the language of Poisson geometry. 

\begin{definition} A Poisson manifold $(M,\pi)$ is \emph{scattering-Poisson} if there exists an oriented hypersurface $Z\subset M$ such that there is a bivector $\pi_{sc}\in\bigwedge^2(^{sc}TM)$ with $\rho (\pi_{sc})=\pi$. 

It is assumed that $\pi_{sc}$ is non-degenerate unless otherwise stated. 
\end{definition} 

By dualizing the scattering-symplectic form as in Proposition \ref{scdarboux}, for any non-degenerate scattering-Poisson manifold $(M,\pi)$, for all $p\in Z$ there exists a coordinate chart $(U,x_1,y_1,\dots,x_n,y_n)$ centered at $p$ such that on $U$, the hypersurface $Z$ is locally defined by $\left\{{x_1=0}\right\}$, and 

\begin{equation}\label{eq:ScPos}\displaystyle{\pi =  x_1^3\dfrac{\partial}{\partial y_1}\wedge\dfrac{\partial}{\partial x_1}+x_1^2\dfrac{\partial}{\partial y_1}\wedge\left(\sum_{i=2}^n y_i\dfrac{\partial}{\partial y_i}+x_i\dfrac{\partial}{\partial x_i}\right)+x_1^2\sum_{i=2}^n\dfrac{\partial}{\partial x_i}\wedge\dfrac{\partial}{\partial y_i}.}\end{equation}
  
 {\bf Symplectic foliations.} Every Poisson structure on a manifold induces a foliation by symplectic manifolds. The symplectic foliation for a non-degenerate scattering Poisson structure contains the open leaves $M\setminus Z$, locally given by $\left\{x<0\right\}$ and $\left\{x>0\right\}$ as depicted in Figure \ref{fig:ScFol}. The individual points of the hyperplane $Z$ are zero dimensional symplectic leaves.

\begin{fig}\label{fig:ScFol} sc-Poisson foliation

\begin{center}\includegraphics[scale=1.3]{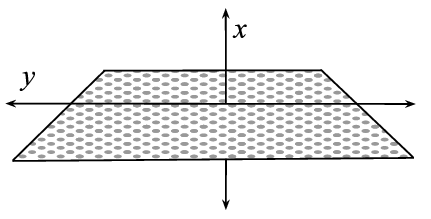}\end{center}\end{fig}

Recall from \cite{guillemin2014symplectic} and \cite{scott2016geometry}, in the case of $b$- and $b^k$- Poisson structures, the symplectic foliation contains the open leaves $\left\{x>0\right\}$ and $\left\{x<0\right\}$, and a regular codimension 1 foliation of the hypersurface $Z$ - depicted in Figure \ref{fig:BkFol} as the level sets $y=c$ for each $c\in\mathbb{R}$. 

\begin{fig} \label{fig:BkFol}  $b^k$-Poisson foliation

\begin{center}\includegraphics[scale=1.2]{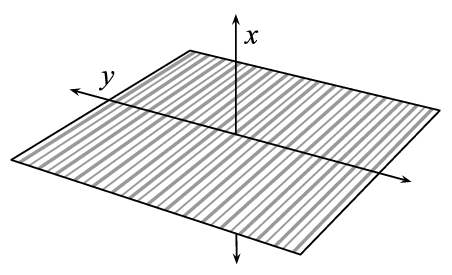}\end{center}

 \end{fig}  
  
Note that the Hamiltonian vector fields associated to a scattering-Poisson manifold are all zero at $Z$.  However, as we will show in the following discussion, $\pi$ contains information about the induced contact structure on the degeneracy hypersurface. Since contact structures are maximally non-integrable, it is fitting that the associated symplectic foliation is maximally trivial. 






{\bf Poisson cohomology.} Consider a non-degenerate scattering Poisson manifold $(M,Z,\pi)$  and let $\omega$ be the corresponding scattering-symplectic form. We will now give an explicit description of the scattering-rigged algebroid. Given a choice of local $Z$ defining function $x$, $\omega$ induces a contact form $\alpha$ on $Z$. 

For local coordinates $(y_1,x_2,y_2,\dots x_m,y_m)$, in $Z$ such that $\displaystyle{\alpha = dy_1+\sum_{i=2}^m x_idy_i}$, the local sections of $\mathcal{R}$ are smooth linear combinations of \begin{equation}\label{eq:ScRig} x_1^3\dfrac{\partial}{\partial x_1}, x_1^3\dfrac{\partial}{\partial y_1},x_1^2\dfrac{\partial}{\partial x_2},\dots,x_1^2\dfrac{\partial}{\partial x_m}, x_1^2\left(x_2\dfrac{\partial}{\partial y_1}-\dfrac{\partial}{\partial y_2}\right),\dots,x_1^2\left(x_m\dfrac{\partial}{\partial y_1}-\dfrac{\partial}{\partial y_m}\right).\end{equation} \noindent Notice that we can locally identify $\Gamma(\mathcal{R}^*)$ as the span of \vspace{2ex}

\centerline{$\dfrac{dx_1}{x_1^3},~~~~~ \dfrac{\alpha}{x_1^3},~~~~~~~~~\dfrac{dy_2}{x_1^2},\dots,\dfrac{dy_m}{x_1^2},~~~~~~~~~\dfrac{dx_2}{x_1^2},\dots, \dfrac{dx_m}{x_1^2}.$}\vspace{3ex}

Further, by Lemma \ref{RiggedLemma}, the Poisson cohomology of a non-degenerate scattering-Poisson manifold $(M,Z,\pi)$ is isomorphic to the Lie algebroid cohomology $^{\mathcal{R}}H^*(M)$ of the scattering rigged algebroid $\mathcal{R}$ as locally identified in equation (\ref{eq:ScRig}).

Given a  $2n$-dimensional non-degenerate scattering Poisson manifold $(M,Z,\pi)$, let $\xi$ denote the contact distribution on $Z$ induced by $\pi$. Let $\alpha$ be a contact one form such that $\ker\alpha =\xi$.  We will use the space of differential forms $\Omega_\xi^{k}(Z)$ defined as  

\centerline{$\Omega_\xi^{k}(Z):=\alpha\wedge \Omega^{k}(Z) \subset \Omega^{k+1}(Z)$}
 
Note that any other choice of contact one-form $\widetilde{\alpha}$ will give us the same space $\Omega_\xi^{k}(Z)$ of forms because contact one forms are conformally equivalent, i.e. $\widetilde{\alpha} = f \alpha$ for a non-zero function $f\in\mathcal{C}^\infty(Z)$.

We will also consider the forms 

\centerline{$\mathcal{K}^{k}:=\ker(d\alpha\wedge:\Omega_\xi^{k}(Z)\to\Omega_\xi^{k+2}(Z))$.}  

Note that this module is independent of the choice of $\alpha$ because any other choice of contact form will give a symplectic structure conformal to $d\alpha$ on $\xi$. We adopt the convention that $\mathcal{K}^{k}=0$ for $k\leq 0$.

\begin{theorem}\label{scpoissoncohom} Given $(M,Z,\pi)$  a $2n$-dimensional non-degenerate scattering Poisson manifold, let $\xi$ denote the contact distribution on $Z$ induced by $\pi$. The Poisson cohomology $H^p_\pi(M)$ of $(M,Z,\pi)$ is 

\centerline{$ H^{p}(M)\oplus H^{p-1}(Z)\oplus\mathcal{J}_Z^1({}^\mathcal{R}\Omega_{cl}^p(M)).$} \noindent
Given a fixed local $Z$ defining function $x$, 

\centerline{$H^p_\pi(M)\simeq H^{p}(M)\oplus H^{p-1}(Z)\oplus\Omega^{p-1}(Z)\oplus\Omega_\xi^{p-1}(Z)\oplus \mathcal{K}^{p-2}.$} \end{theorem}

\begin{remark} Given a contact hypersurface $(Z,\alpha)$ of dimension $2n+1$, the map $d\alpha\wedge:\Omega_\xi^{k}(Z)\to\Omega_\xi^{k+2}(Z)$, by a local computation at a point, is injective  for $k=0,\dots,n-1$ and is identically equal to zero for $k=2n, 2n+1$. Thus $\mathcal{K}^{k}=0$ for $k=0,\dots,n-1$ and $\mathcal{K}^{k}=\Omega_\xi^{k}(Z)$ for $k=2n, 2n+1$. 
\end{remark}

\begin{proof} We are left to compute ${}^\mathcal{R}H^p(M)$. We have a short exact sequence \begin{center}$0\to {}^b\Omega^{k}(M)\xrightarrow{i^*}{}^\mathcal{R}\Omega^{k}(M)\xrightarrow{P} \mathscr{C}^k\to 0$ \end{center} where 

\centerline{$\mathscr{C}^p={^\mathcal{R}\Omega^{p}(M)}/{^b\Omega^{p}(M)}$} \noindent is the quotient under the evaluation bundle map $i: {}^\mathcal{R}TM\to {}^bTM$. We have a differential ${}^\mathscr{C}d$ induced by the differential $^\mathcal{R}d$ on ${}^\mathcal{R}\Omega^{p}(M)$. In particular, if $P$ is the projection $^\mathcal{R}\Omega^{p}(M)\to{}^\mathcal{R}\Omega^{p}(M)/{}^b\Omega^{p}(M),$ then  $^{\mathscr{C}}d(\eta)=P(^\mathcal{R}d(\theta))$ where  $\theta\in{^\mathcal{R}\Omega^{p}(M)}$ is any form such that $P(\theta)=\eta$. Hence $(^{\mathscr{C}}d)^2=0$ and $(\mathscr{C}^{*},{^{\mathscr{C}}d})$ is in fact a complex.  

Given a tubular neighborhood $\tau=Z\times(-\varepsilon,\varepsilon)_x$ of $M$ near $Z$, we can write a degree-$k$ form $\nu$ in $^\mathcal{R}\Omega^k(M)$ as an expansion in $x$ of the form

\centerline{$\displaystyle{\nu = \mu_b+ \dfrac{dx}{x^{2k+1}}\wedge(\sum_{i=0}^{2k-1}\eta_ix^i)+\dfrac{1}{x^{2k}}\sum_{i=0}^{2k-1}\beta_ix^i+\dfrac{dx\wedge\alpha\wedge\theta}{x^{2k+2}} + \dfrac{\alpha\wedge\gamma}{x^{2k+1}},}$} \noindent $\mu_b$ is a b-form, $\alpha$ is the contact form on $Z$ induced by $\omega$ and $x$, $\eta_i\in\Omega^{k-1}(Z)$, $\beta_i\in\Omega^{k}(Z)$, $\alpha\wedge\theta\in\Omega^{k-1}(Z)$, and $\alpha\wedge\gamma\in\Omega^{k}(Z)$.

We write $R_b(\nu)=\mu_b$ and $S_b(\nu)=\nu-R_b(\nu)$ for `regular' and `singular' parts. It is easy to see that $R_b({}^{\mathcal{R}}d\nu)={}^{\mathcal{R}}d(R_b(\nu))$ and $S_b({}^{\mathcal{R}}d\nu)={}^{\mathcal{R}}d(S_b(\nu))$. Thus the trivialization $\tau$ induces a splitting ${}^{\mathcal{R}}\Omega^*(M)={}^b\Omega^*(M)\oplus\mathscr{C}^*$ as complexes. As a consequence ${}^{\mathcal{R}}H^k(M)={}^bH^k(M)\oplus  H^k(\mathscr{C}^{*})$ and we are left to compute the cohomology of the quotient complex.  

We have that 

\centerline{$\displaystyle{{}^{\mathcal{R}}d(S_b(\nu))=-\sum_{i=0}^{2k-1}\dfrac{dx}{x^{2k+1}}\wedge d\eta_ix^i- \sum_{i=0}^{2k-1}\frac{(2k-i)dx}{x^{2k+1}}\wedge\beta_ix^i+\sum_{i=0}^{2k-1}\dfrac{d\beta_i}{x^{2k}}x^i}$}

\centerline{$\displaystyle{-\dfrac{dx\wedge d\alpha\wedge\theta}{x^{2k+2}}+ \dfrac{dx\wedge \alpha\wedge d\theta}{x^{2k+2}} -\dfrac{2k+1}{x^{2k+2}}dx\wedge \alpha\wedge\gamma + \dfrac{d\alpha\wedge\gamma}{x^{2k+1}}-\dfrac{\alpha\wedge d\gamma}{x^{2k+1}}.}$}

Thus the kernel $^{\mathscr{C}}d:\mathscr{C}^k\to\mathscr{C}^{k+1}$ is defined by the relation \begin{center}$-d\eta_ix^{i+1}-(2k-i)\beta_ix^{i+1}-d\alpha\wedge\theta+\alpha\wedge d\theta -(2k+1)\alpha\wedge\gamma=0.$ \end{center} 

In order for the expression to be zero, the coefficients of the polynomial must be zero and thus \centerline{$\beta_i=\dfrac{-d\eta_i}{(2k-i)}$} for $i=0,\dots,2k-1$. 

Now we consider the kernel relation given by the coefficient 

\centerline{$-d\alpha\wedge\theta+\alpha\wedge d\theta-(2k+1)\alpha\wedge\gamma=0.$} \noindent By contracting with $R$, the Reeb vector field associated to $\alpha$, we recover 

\centerline{$d\theta-\alpha\wedge i_R d\theta-(2k+1)\gamma =0.$} \noindent Thus 

\centerline{$\gamma = \dfrac{d\theta-\alpha\wedge i_R d\theta}{(2k+1)}.$} \noindent Substituting this into the original expression, we have that 

\centerline{$-d\alpha\wedge\theta+\alpha\wedge d\theta-(2k+1)\alpha\wedge\dfrac{d\theta}{(2k+1)}=0$} \noindent since $\alpha^2=0$. Thus $d\alpha\wedge\theta=0$. 

Thus all closed forms in $\mathscr{C}^{k}$ are of the form 

\centerline{$\displaystyle{ \dfrac{dx}{x^{2k+1}}\wedge(\sum_{i=0}^{2k-1}\eta_ix^i)+\dfrac{1}{x^{2k}}\sum_{i=0}^{2k-1}\dfrac{-d\eta_i}{(2k-i)}x^i+\dfrac{dx\wedge\alpha\wedge\theta}{x^{2k+2}} + \dfrac{\alpha}{x^{2k+1}}\wedge\dfrac{(d\theta-\alpha\wedge i_R d\theta)}{(2k+1)},}$} \noindent where $\theta\in\ker (d\alpha\wedge:\Omega_\xi^{k-2}(Z)\to\Omega_\xi^{k}(Z))$.  

Elements in ${}^\mathscr{C}d(\mathscr{C}^{k-1})$ are of the form 

 \begin{equation}\label{eq:scriggedcohom} -\sum_{i=0}^{2k-3}\dfrac{dx}{x^{2k-1}}\wedge d\eta_ix^i- \sum_{i=0}^{2k-3}\frac{(2k-2-i)dx}{x^{2k-1}}\wedge\beta_ix^i+\sum_{i=0}^{2k-3}\dfrac{d\beta_i}{x^{2k-2}}x^i \end{equation}  
 
\centerline{$-\dfrac{dx\wedge d\alpha\wedge\theta}{x^{2k}}+ \dfrac{dx\wedge \alpha\wedge d\theta}{x^{2k}} -\dfrac{2k-1}{x^{2k}}dx\wedge \alpha\wedge\gamma + \dfrac{d\alpha\wedge\gamma}{x^{2k-1}}-\dfrac{\alpha\wedge d\gamma}{x^{2k-1}}.$}

\noindent Thus there is the element 

\centerline{$\displaystyle{\sum_{j=2}^{2k-1}\dfrac{-x^{j-2}\eta_{j}}{(2k-j)x^{2k-2}}}$} \noindent in $\mathscr{C}^{k-1}$ such that 

\centerline{$\displaystyle{d\left(\sum_{j=2}^{2k-1}\dfrac{-x^{j-2}\eta_{j}}{(2k-j)x^{2k-2}}\right)=\sum_{j=2}^{2k-1}\left(\dfrac{dx}{x^{2k+1}}\wedge\eta_{j}x^{j}-\dfrac{d\eta_{j}x^{j}}{(2k-j)x^{2k}} \right).}$} \noindent If we express $\eta_i=\delta_i+\alpha\wedge\gamma_i$ for $\delta_i,\gamma_i\in\Omega_\xi^k(Z),$ then there is the element 

\centerline{$\dfrac{-\alpha\wedge\gamma_1}{(2k-1)x^{2k-1}}$} \noindent in $\mathscr{C}^{k-1}$ such that 

\centerline{$d\left(\dfrac{-\alpha\wedge\gamma_1}{(2k-1)x^{2k-1}}\right)= \dfrac{dx\wedge\alpha\wedge\gamma_1x}{x^{2k+1}}-\dfrac{d(\alpha\wedge\gamma_1)x}{(2k-1)x^{2k}}.$}

By (\ref{eq:scriggedcohom}), the remaining terms in a closed form in $\mathscr{C}^k$ are too singular to appear in the image $d(\mathscr{C}^{k-1})$. Thus an element of $H^k(\mathscr{C})$ has a representative of the form  

\centerline{$\displaystyle{\nu=\dfrac{dx}{x^{2k+1}}\wedge(\delta_0+\alpha\wedge\gamma_0)+\dfrac{dx}{x^{2k+1}}\wedge x\delta_1+ \dfrac{dx}{x^{2k+2}}\wedge\alpha\wedge\theta}$} 

\centerline{$\displaystyle{-\dfrac{d(\delta_0+\alpha\wedge\gamma_0)}{(2k)x^{2k}}-\dfrac{d\delta_1}{(2k-1)x^{2k-1}}-\dfrac{d(\alpha\wedge\theta)}{(2k+1)x^{2k+1}}}$} \noindent  where $\delta_0, \gamma_0,\delta_1\in\Omega_\xi^{*}(Z)$ and $\theta\in\mathcal{K}^{k-2}=\ker(d\alpha\wedge:\Omega^{k-2}_\xi(Z)\to\Omega^{k}_\xi(Z))$ and each such form represents a separate cohomology class. Thus for a fixed local $Z$ defining function $x$, by the map  

\centerline{$\nu\mapsto ( \delta_0+\alpha\wedge\gamma_0, \delta_1,\theta ),$} 

\centerline{$H^k(\mathscr{C})\simeq\Omega^{k-1}(Z)\oplus\Omega_\xi^{k-1}(Z)\oplus \mathcal{K}^{k-2}.$}

To conclude, we consider what happens under change of local $Z$ defining function $x$. Note that $\nu$ is completely determined by $i_{\partial x}\nu$. Thus, to account for change in defining function for the contribution of terms $$\dfrac{dx}{x^{2k+1}}\wedge(\delta_0+\alpha\wedge\gamma_0),$$ we tensor $\Omega^{k-1}(Z)$ with $\Gamma((N^\ast Z)^{-2k})$ over smooth real valued functions on $Z$: $\Omega^{k-1}(Z; (N^\ast Z)^{-2k})$. Similarly, the contribution of terms 
$\dfrac{dx}{x^{2k+1}}\wedge x\delta_1$ is $\Omega^{k-1}_{\xi} (Z; (N^\ast Z)^{-2k+1})$ and the contribution of terms $\dfrac{dx}{x^{2k+2}}\wedge\alpha\wedge\theta$ is $\mathcal{K}^{k-2}\otimes_{\mathcal{C}^\infty(Z)} \Gamma((N^\ast Z)^{-2k-1})$.

We have shown that the Poisson cohomology $H^k_\pi(M)$ of $(M,Z,\pi)$ is 

\centerline{${}^{b}H^k(M)\oplus\Omega^{k-1}(Z; (N^\ast Z)^{-2k})\oplus \Omega^{k-1}_{\xi} (Z; (N^\ast Z)^{-2k+1}) \oplus \mathcal{K}^{k-2}\otimes_{\mathcal{C}^\infty(Z)} \Gamma((N^\ast Z)^{-2k-1})$} \noindent and, given a fixed local $Z$ defining function $x$, is 

\centerline{${}^{b}H^k(M)\oplus \Omega^{k-1}(Z)\oplus\Omega_\xi^{k-1}(Z)\oplus \mathcal{K}^{k-2}.$} \noindent The final isomorphism is due to Mazzeo-Melrose that $^{b}H^k(M)\simeq H^k(M)\oplus H^{k-1}(Z)$. \end{proof}

\bibliographystyle{plain}
\bibliography{reference}

\end{document}